\documentclass{amsart}

\calclayout

\usepackage{wasysym}
\usepackage{latexsym,amssymb}
\usepackage{enumerate, verbatim, mathrsfs} 
\usepackage{graphics}
\usepackage{hyperref} 
\usepackage[all]{hypcap}
\usepackage{amsthm}
\usepackage{amsfonts}
\usepackage{longtable}
 \usepackage{epsfig}
 \usepackage{amsmath}
\usepackage{hhline}
 \usepackage{xfrac}
 \usepackage{yfonts}
\usepackage{mathrsfs}
 \usepackage[usenames, dvipsnames]{xcolor}
 \usepackage{mdframed}
 \usepackage{mathtools}
\usepackage[final]{showkeys}
\usepackage[final]{todonotes}

\definecolor{cadmiumred}{rgb}{0.89, 0.0, 0.13}
 
\usepackage{tikz}
\usetikzlibrary{calc, babel}
\usetikzlibrary{datavisualization}
\usetikzlibrary{datavisualization.formats.functions}
\usetikzlibrary{decorations.markings}

\newcommand{\SL}{\text{SL}}
\newcommand{\Mp}{\text{Mp}}
 \newcommand{\lev}{\text{lev}}
\newcommand{\uL}{\underline{L}} \newcommand{\rk}{\text{rk}(\uL)}

\newcommand{\iso}{\text{Iso}(D_{\underline{L}})}

\newcommand{\ord}{\text{ord}}

\newcommand{\Aut}{\text{Aut}}
\newcommand{\tT}{\tilde{T}}
\newcommand{\tS}{\tilde{S}}
\newcommand{\e}{\mathfrak{e}}
\newcommand{\CL}{\mathbb{C}[L^\#/L]}
\newcommand{\Av}{\text{Av}}
 
\theoremstyle{plain}
\newtheorem{theorem}{Theorem}
\newtheorem{lemma}[theorem]{Lemma}
\newtheorem{proposition}[theorem]{Proposition}
\newtheorem{corollary}[theorem]{Corollary}

\theoremstyle{definition}

\newtheorem{example}[theorem]{Example}

\theoremstyle{remark}
\newtheorem{remark}[theorem]{Remark}

\begin{document}

\tikzset{->-/.style={decoration={
  markings,
  mark=at position .5 with {\arrow{>}}},postaction={decorate}}}
\tikzset{-->/.style={decoration={
  markings,
  mark=at position .75 with {\arrow{>}}},postaction={decorate}}}
  \tikzset{<--/.style={decoration={
  markings,
  mark=at position .75 with {\arrow{<}}},postaction={decorate}}}

\title[Jacobi--Poincar\'e and Eisenstein series]{Poincar\'e and Eisenstein series \\ for Jacobi forms of lattice index}
\author{Andreea Mocanu}
\address{The University of Nottingham}
\email{andreea.mocanu1@nottingham.ac.uk}

\begin{abstract}
Poincar\'e and Eisenstein series are building blocks for every type of modular forms. We define Poincar\'e series for Jacobi forms of lattice index and state some of their basic properties. We compute the Fourier expansions of Poincar\'e and Eisenstein series and  give an explicit formula for the Fourier coefficients of the trivial Eisenstein series. For even weight and fixed index, finite linear combinations of Fourier coefficients of non-trivial Eisenstein series are equal to finite linear combinations of Fourier coefficients of the trivial one.
\end{abstract}

\maketitle

\tableofcontents

\section{Introduction}

Jacobi forms arise naturally in number theory in several ways: theta functions arise as functions of lattices (see \cite{J}) and Siegel modular forms give rise to Jacobi forms through their Fourier--Jacobi expansion (see \cite{PS}), for example. The arithmetic theory of Jacobi forms of scalar index was developed in the 1980's in Eichler and Zagier's monograph, \cite{EZ}. Several generalizations of this type of Jacobi forms have been studied in detail since then, such as Siegel--Jacobi forms (see \cite{Z}), Jacobi forms of lattice index (see \cite{G}) and Jacobi forms over number fields (see \cite{Bo}). The popularity of Jacobi forms has increased in recent years due to their connection to string theory. They play a part in the Mirror Symmetry conjecture for K$3$ surfaces (see \cite{GN}), they determine Lorentzian Kac--Moody Lie (super) algebras of Borcherds type (see \cite{G2}) and a certain type of Jacobi forms can be the elliptic genus of Calabi--Yau manifolds (see \cite{G1}). In this paper, we work with Jacobi forms of lattice index (also referred to as ``Jacobi forms in many variables'' in the literature), which appear for example in the theory of reflective modular forms (see \cite{G3}) and vertex operator algebras (see \cite{KM}). 


Eisenstein and Poincar\'e series are the most simple examples of modular forms. They are obtained by taking the average of a function over a group (modulo a parabolic subgroup). This construction makes them automatically invariant under the group action. In the context of elliptic 
 modular forms, they satisfy the important property of reproducing Fourier coefficients of cusp forms under a suitably defined scalar product. Furthermore, Poincar\'e series generate the space of cusp forms and they can be used for instance to construct lifting maps between different types of modular forms. It is well-known that Eisenstein series are orthogonal to cusp forms in the case of elliptic modular forms. Let $M_k$ denote the space of modular forms of weight $k$ for the modular group $\SL_2(\mathbb{Z})$, let $S_k$ be its subspace of cusp forms and let $M_k^{Eis}$ be the spanning set of the Eisenstein series of weight $k$; we obtain the following decomposition:
\begin{equation*}
M_k=S_k\oplus M_k^{Eis}.
\end{equation*}
Thus, Eisenstein and Poincar\'e series describe a fixed space of modular forms completely and it is important to have explicit formulas for their Fourier expansions. 


Similar constructions hold for other types of modular forms and the purpose of this paper is to generalize these results for Jacobi forms of lattice index. To the best of the author's knowledge, Poincar\'e series have not been defined in the literature in this context. We show that they are cusp forms and that they reproduce the Fourier coefficients of other cusp forms under the Petersson scalar product; we compute their Fourier expansion, expressing their Fourier coefficients as infinite series (see Theorem \ref{T:Poincare}). The definition of Jacobi--Eisenstein series was given for instance in \cite{A}, where some of their properties were studied (such as dimension formulas for their spanning set and the fact that they are Hecke eigenforms). We show that they are orthogonal to cusp forms and compute their Fourier expansion (see Theorem \ref{T:Eisenstein}). We give an explicit formula for the Fourier coefficients of the trivial Eisenstein series (see Theorem \ref{T:E0}) and obtain a linear relation between trivial and non-trivial Eisenstein series (see Proposition \ref{P:avxiso}).

The following section contains the notation and theory that are necessary in order to make the main results precise. Sections \ref{S:P} and \ref{S:E} are dedicated to the proofs of Theorems \ref{T:Poincare} and \ref{T:Eisenstein}, respectively. In the proofs of the two main theorems, we use the approach B{\"o}cherer and Kohnen employed in their work on Siegel modular forms \cite{BK}. 
In Section \ref{S:E0}, we give an explicit formula for the Fourier coefficients of the trivial Eisenstein series. This formula involves classical number theoretical objects such as Bernoulli numbers and values of Dirichlet $L$-functions at negative integers. 
We use results of Bruinier and Kuss from \cite{BK} on $L$-series arising from representation numbers of quadratic forms. Finally, in Section \ref{S:Er} we obtain a linear relation between trivial and non-trivial Eisenstein series in the case of even weight. We use the existence of an isomorphism between spaces of Jacobi forms and spaces of vector-valued modular forms and a linear operator which was defined by Williams in \cite{Wi} for vector-valued modular forms. The proofs in this section rely heavily on the connection between the Weil and the Schr\"odinger representations.

\section{Notation and elementary results}\label{S:ND}

Let $\mathbb{Q}_{\leq0}$ denote the set of non-positive rational numbers and let $\mathbb{R}_{\geq0}$ denote the set of non-negative real numbers. Denote the finite group of residue classes modulo $m$ by $\mathbb{Z}_m$. Always consider the principal branch of the complex square root, i.e. having argument in $(-\pi/2,\pi/2]$, unless stated otherwise. For an integer $n$, let $e_n(x)$ denote the function $e^{2\pi ix/n}$ and write $e(x)$ for $e_1(x)$. For a rational number $t$, let $\lceil t\rceil$ denote the smallest integer that is greater than or equal to $t$ and let $\lfloor t\rfloor$ denote the greatest integer that is smaller than or equal to $t$. For each prime $p$, let $\ord_p$ be the $p$-adic valuation on $\mathbb{Q}$. Let $\mathfrak{H}$ denote the upper half-plane $\{\tau\in\mathbb{C}:\Im(\tau)>0\}$. Let $\left(\frac{\cdot}{\cdot}\right)$ denote the Kronecker symbol and, for every discriminant $\mathfrak{f}$, write $\chi_{\mathfrak{f}}(\cdot)$ for the quadratic Dirichlet character modulo $|\mathfrak{f}|$ given by $\chi_{\mathfrak{f}}(a)=\left(\frac{\mathfrak{f}}{a}\right)$. Let $\mu$ denote the M\"obius function, let $\zeta$ denote the Riemann zeta function, let $\sigma_t$ denote the $t$-th divisor sum and, for every Dirichlet character $\chi$, let $\sigma_t^{\chi}$ denote the twisted divisor sum $\sigma_t^{\chi}(n)=\sum_{d\mid n}\chi(d)d^t$. Let $B_m$ denote the $m$-th Bernoulli number and let $B_n(\cdot)$ denote the $n$-th Bernoulli polynomial, which is defined as
\begin{equation*}
B_n(x):=\sum_{j=0}^n\binom{n}{j}B_{n-j}x^j.
\end{equation*}

Recall that the $J$-Bessel function of index $\alpha>0$ is defined by the following series expansion around $x=0$:
\begin{equation}\label{eq:Jbessel}
 J_{\alpha}(x):=\sum_{n=0}^{\infty}\frac{(-1)^n}{n!\Gamma(n+\alpha+1)}\left(\frac{x}{2}\right)^{2n+\alpha}.
 \end{equation}

\subsection{Lattices}\label{notation}

Let $L$ and $N$ be $\mathbb{Z}$-modules, with $L$ free of finite rank equal to $m$. A map $\beta:L\times L\to N$ is called a symmetric $\mathbb{Z}$-bilinear form if
\begin{equation*}
\begin{split} 
\beta(x,y)&=\beta(y,x),\\
\beta( x,\lambda y+\mu z)&=\lambda\beta(x,y)+\mu\beta(x,z),
\end{split}
\end{equation*}
for all $x,y,z$ in $L$ and all $\lambda,\mu$ in $ \mathbb{Z}$. If $N=\mathbb{Z}$, then $\beta$ is said to be integral. The bilinear form $\beta$ is called non-degenerate if the map $x\mapsto\beta(\cdot,x)$ from $L$ to Hom$(L,N)$ is injective. Let $\{e_1,\dots,e_m\}$ be a $\mathbb{Z}$-basis of $L$. The matrix $G=(\beta(e_i,e_j))_{i,j}$ in $ M_m(N)$ is called the Gram matrix of $\beta$ with respect to $\{e_1,\dots,e_m\}$. Identify every element in the lattice with its coefficient vector and write
\begin{equation*}
\beta(x,y)=x^tGy,
\end{equation*}
where $A^t$ denotes the transpose of a matrix $A$. It is therefore possible to extend the domain of definition of $\beta$ to $L\otimes_{\mathbb{Z}}\mathbb{Q}$, $L\otimes_{\mathbb{Z}}\mathbb{R}$ and $L\otimes_{\mathbb{Z}}\mathbb{C}$ in a natural way via the above matrix formula. 

A {\em lattice} $\uL$ over $\mathbb{Z}$ is a pair $(L,\beta)$ such that $L$ is a free $\mathbb{Z}$-module of finite rank and $\beta:L\times L\to N$ is a symmetric, non-degenerate $\mathbb{Z}$-bilinear form. The lattice $\uL$ is said to be integral if the associated bilinear form is integral. The rank of $\uL$, denoted by $\rk$, is defined as the rank of $L$ as a $\mathbb{Z}$-module. An integral lattice $\uL=(L,\beta)$ is called positive-definite if $\beta(x,x)>0$ for all $x$ in $ L$ such that $x\neq0$. It is called even if $\beta(x,x)$ is even for all $x$ in $ L$, otherwise it is called odd. By abuse of notation, denote the quadratic form associate to $\uL$ by $\beta(x)$, i.e. $\beta(x)=\frac{1}{2}\beta(x,x)$ for all $x$ in $ L$. Note that these definitions are still valid if $\mathbb{Z}$ is replaced with any commutative ring with identity (such as the ring of integers of a number field or the $p$-adic integers, for example). 

Let $\uL=(L,\beta)$ be an integral lattice and define the following set:
\begin{equation}\nonumber 
L^\#:=\{y\in L\otimes_{\mathbb{Z}}\mathbb{Q}:\beta(y,x)\in\mathbb{Z}\text{ for all }x\text{ in } L\}.
\end{equation} 
The {\em dual lattice} of $\uL$ is defined as the pair $\underline{L}^\#=(L^\#,\beta)$. The group $L$ is clearly a subgroup of $L^\#$ and $L^\#/L$ is an abelian group of finite order. The order of this group is called the determinant of $\uL$, denoted by $\det(\uL):=|L^\#/L|=|\det(G)|$.  If $\uL$ is even, then the reduction of $\beta$ modulo $\mathbb{Z}$ induces a bilinear form on $L^\#/L$. In this case, define the discriminant form associated with $\uL$ as the pair
\begin{equation}\nonumber
D_{\uL}:=\left(L^\#/L,x+L\mapsto\beta(x)+\mathbb{Z}\right) 
\end{equation}
and the isotropy set of $D_{\uL}$ as the set
\begin{equation}\nonumber 
\iso:=\{x\in L^\#/L:\beta(x)\in\mathbb{Z}\}. 
\end{equation} 
For every $x\in D_{\uL}$, let $N_x$ denote its order in $L^\#/L$. The level of $\uL$, denoted by $\lev(\uL)$, is the smallest positive integer that satisfies \lev$(\uL)\cdot\beta(x)\in\mathbb{Z}$ for all $x$ in $ L^\#$. Set
\begin{equation}\label{eq:Delta}
\Delta(\uL):= 
\begin{cases}
(-1)^{\frac{\rk}{2}}\text{det}(\uL), &\text{if } \rk\equiv0 \bmod2 \text{ and }\\
(-1)^{\lfloor\frac{\rk}{2}\rfloor}2\text{det}(\uL), &\text{if } \rk\equiv1 \bmod2. 
\end{cases} 
\end{equation} 
For $a$ in $\mathbb{\mathbb{Z}}$ and $D$ in $\mathbb{Q}$ such that $D\cdot\Delta(\uL)\in\mathbb{Z}$, define
\begin{equation}\label{eq:chiL}
\begin{split}
\chi_{\uL}(D,a):=\left(\frac{D\Delta(\uL)}{a}\right).
\end{split}
\end{equation} 
It is known that $\Delta(\uL)$ is a discriminant, i.e. it is congruent to $0$ or $1$ modulo $4$ (see Lemma $14.3.20$ and Remark $14.3.23$ in $\S14.3$ of \cite{CS}). 

\subsection{Jacobi forms}

In this subsection we give a brief overview of the theory of Jacobi forms of lattice index, following the exposition in \cite{A}. For details and proofs, the reader can consult \cite{A}.

Let $\uL=(L,\beta)$ be an even lattice. The integral Heisenberg group associated with $\underline{L}$ is defined as
\begin{equation*}
H_{\underline{L}}(\mathbb{Z}):=\{(x,y):x,y\in L\}, 
\end{equation*} 
with composition law given by component wise addition, i.e.
\begin{equation*}
(x_1,y_1)(x_2,y_2):=\left(x_1+x_2,y_1+y_2\right).
\end{equation*} 
The group $\SL_2(\mathbb{Z})$ of $2\times2$ integral matrices with determinant equal to $1$ acts on $H_{\underline{L}}(\mathbb{Z})$ from the right via:
\begin{equation*}
\left(\left(x,y\right),A\right)\mapsto(x,y)^A:=\left(ax+cy,bx+dy\right),
\end{equation*} 
where $A=\left(\begin{smallmatrix}
a&b\\c&d
\end{smallmatrix}\right)\in\SL_2(\mathbb{Z})$. The {\em Jacobi group associated with $\uL$} is the semi-direct product of $\SL_2(\mathbb{Z})$ and $H_{\underline{L}}(\mathbb{Z})$,
\begin{equation*}
J_{\underline{L}}(\mathbb{Z}):=\text{SL}_2(\mathbb{Z})\ltimes
H_{\underline{L}}(\mathbb{Z}),
\end{equation*}
with the following composition law:
\begin{equation*}
 (A,h)(A',h')=(AA',h^{A'}h'),
\end{equation*}
for every $A,A'$ in $\SL_2(\mathbb{Z})$ and every $h,h'$ in $ H_{\uL}(\mathbb{Z})$. Note that it is also possible to define a rational Jacobi group $J_{\uL}(\mathbb{Q})$ and a real Jacobi group $J_{\uL}(\mathbb{R})$ in an analogous way.

It was shown in $\S2.2$ of \cite{A} that the Jacobi group acts on the left on the space $\mathfrak{H}\times(L\otimes_{\mathbb{Z}}\mathbb{C})$. If $A=\left(\begin{smallmatrix}
a&b\\c&d
\end{smallmatrix}\right)\in\SL_2(\mathbb{Z})$ and $h=(x,y)\in H_{\uL}(\mathbb{Z})$, then the action of $(A,h)$ on $\mathfrak{H}\times(L\otimes_{\mathbb{Z}}\mathbb{C})$ is defined as
\begin{equation*}
\left((A,h),(\tau,z)\right)\mapsto(A,h)(\tau,z):=\left(A\tau,\frac{z+x\tau+y}{c\tau+d}\right),
\end{equation*}
for each pair $(\tau,z)$ in $\mathfrak{H}\times(L\otimes_{\mathbb{Z}}\mathbb{C})$. The Jacobi group also acts on the right on the space of holomorphic functions defined on $\mathfrak{H}\times(L\otimes_{\mathbb{Z}}\mathbb{C})$ and taking values in $\mathbb{C}$. For each such function $\phi$ and for every $A=\left( \begin{smallmatrix}
a&b\\c&d \end{smallmatrix} \right)$ in $\text{SL}_2(\mathbb{Z})$, define the following action:
\begin{equation}\nonumber
\phi|_{k,\underline{L}}A(\tau,z):=\phi\left(A\tau,\frac{z}{c\tau+d}\right)(c\tau+d)^{-k}e\left(\frac{-c\beta(z)}{c\tau+d}\right)
\end{equation}
and, for every $h=(x,y)$ in $ H_{\underline{L}}(\mathbb{Z})$, define the following action:
\begin{equation}\nonumber
\phi|_{k,\underline{L}}h(\tau,z):=\phi(\tau,z+x\tau+y)e\left(\tau\beta(x)+\beta(x,z)\right).
\end{equation}
The action of $J_{\uL}(\mathbb{Z})$ is defined as
\begin{equation}\nonumber
\left(\phi,(A,h)\right)\mapsto\phi|_{k,\underline{L}}(A,h):=(\phi|_{k,\underline{L}}A)|_{k,\underline{L}}h.
\end{equation} 
We leave it to the reader to verify that this is indeed a group action.

Let $k$ be a positive integer and let $\uL=(L,\beta)$ be a positive-definite, even lattice. A {\em Jacobi form of weight $k$ and index $\underline{L}$} is a holomorphic function $\phi:\mathfrak{H}\times(L\otimes_{\mathbb{Z}}\mathbb{C})\to\mathbb{C}$ with the following properties: 
\begin{itemize}
\item For all $\gamma$ in $ J_{\uL}(\mathbb{Z})$, the following holds: 
\begin{equation*}
\phi|_{k,\underline{L}}\gamma(\tau,z)=\phi(\tau,z).
\end{equation*}
\item The function $\phi$ has a Fourier expansion of the form
\begin{equation}\label{eq:Fclassical}
\phi(\tau,z)=\sum_{\substack{n\in\mathbb{Z},r\in
L^\#\\n\geq\beta(r)}}c(n,r)e\left(n\tau+\beta(r,z)\right). 
\end{equation}
\end{itemize}
For fixed weight and index, denote the $\mathbb{C}$-vector space consisting of all such functions by $J_{k,\underline{L}}$. For example, if we take $m$ in $\mathbb{N}$ and consider the lattice $\underline{L}=(\mathbb{Z},(x,y)\mapsto2mxy)$, then the space $J_{k,\underline{L}}$ is the space $J_{k,m}$ of Jacobi forms of weight $k$ and index $m$ defined in \cite{EZ}. Similarly, if we take $G$ to be a $g\times g$ positive-definite symmetric integral matrix with even diagonal elements and we consider the lattice $\underline{L}=\left(\mathbb{Z}^{(g,1)},(x,y)\mapsto x^tGy\right)$, then $J_{k,\underline{L}}$ is the space $J_{k,\frac{1}{2}G}$ of Jacobi forms of weight $k$ and index $\frac{1}{2}G$ defined in \cite{BK}.

The reader will notice that the Fourier expansions in our main theorems do not look like the one in \eqref{eq:Fclassical} . This is due to the following very useful fact, which was proven in $\S2.4$ of \cite{A}:

\begin{proposition}\label{P:JFcoeffs}
If $\phi$ in $ J_{k,\uL}$ has a Fourier expansion of the form
\begin{equation*}
\phi(\tau,z)=\sum_{\substack{n\in\mathbb{Z},r\in
L^\#\\n\geq\beta(r)}}c(n,r)e\left(n\tau+\beta(r,z)\right),
\end{equation*}
then the Fourier coefficients $c(n,r)$ depend only on $n-\beta(r)$ and on $r\bmod L$. More precisely, we have $c(n,r)=c(n',r')$ whenever $r\equiv r'\bmod L$ and $n-\beta(r)=n'-\beta(r')$. 
\end{proposition}
Define the following set, called the support of $\uL$:
\begin{equation}\nonumber
\text{supp}(\uL):=\{(D,r):r\in L^\#, D\in\mathbb{Q}_{\leq0}, D\equiv\beta(r)\bmod\mathbb{Z}\}.
\end{equation}
For every $\phi$ in $ J_{k,\uL}$ with Fourier expansion \eqref{eq:Fclassical} and for each pair $(D,r)$ in $\text{supp}(\uL)$, let $C(D,r):=c\left(\beta(r)-D,r\right)$. Proposition \ref{P:JFcoeffs} implies that every $\phi$ in $ J_{k,\uL}$ has a Fourier expansion of the form
\begin{equation}\label{eq:jacobifourier}
\phi(\tau,z)=\sum_{(D,r)\in\text{supp}(\uL)}C(D,r)e\left((\beta(r)-D)\tau+\beta(r,z)\right).
\end{equation}
We will often use the interplay between these two Fourier expansions. The latter version is used to define Jacobi {\em cusp forms}: a Jacobi form $\phi$ in $ J_{k,\uL}$ is called a cusp form if $C(0,r)=0$ for all $r$ such that
$\beta(r)\in\mathbb{Z}$. The subspace of cusp forms in $J_{k,\uL}$ is denoted by $S_{k,\underline{L}}$. For each $\phi$ in $ J_{k,\uL}$, define its singular term as
 \begin{equation*}
  C_0(\phi)(\tau,z):=\sum_{\substack{r\in L^\#\\ \beta(r)\in\mathbb{Z}}}C(0,r)e\left(\tau\beta(r)+\beta(r,z)\right).
 \end{equation*}

Set 
\begin{equation}\label{eq:Jinf}
J_{\uL}(\mathbb{Z})_{\infty}:=\left\{\left(
\begin{pmatrix} 1&n\\0&1 \end{pmatrix},(0,\mu)\right):n\in\mathbb{Z},\mu\in L\right\}.
\end{equation}
This is the stabilizer of the constant function equal to one, i.e. the set
\begin{equation*}
\{\gamma\in J_{\uL}(\mathbb{Z}):1|_{k,\uL}\gamma=1\}.
\end{equation*}

For $\tau$ in $\mathfrak{H}$ and $z$ in $ L\otimes_{\mathbb{Z}}\mathbb{C}$, let $\tau=u+iv$ and $z=x+iy$. In $\S3.2$ of \cite{A}, the following $J_{\uL}(\mathbb{Z})$-invariant volume element on $\mathfrak{H}\times(L\otimes_{\mathbb{Z}}\mathbb{C})$ is defined:
\begin{equation}\nonumber 
dV_{\uL}(\tau,z):=v^{-\rk-2}dudvdxdy.
\end{equation} 
For two functions $\phi$ and $\psi$ that are invariant under the $|_{k,\uL}$ action of $J_{\uL}(\mathbb{Z})$, define
\begin{equation}\label{eq:omega}
\omega_{\phi,\psi}(\tau,z):=\phi(\tau,z)\overline{\psi(\tau,z)}v^{k}e^{-4\pi\beta(y)v^{-1}}.
\end{equation} 
It is easy to show that the function $\omega_{\phi,\psi}$ is also $J_{\uL}(\mathbb{Z})$-invariant. Given a fundamental domain $\mathfrak{F}$ for the action of $\SL_2(\mathbb{Z})$ on $\mathfrak{H}$ and a fundamental parallelotope  $\mathscr{M}_{\uL}$ for $(L\otimes\mathbb{C})/(\tau L+L)$, choose as a fundamental domain for the action of $J_{\uL}(\mathbb{Z})$ on $\mathfrak{H}\times(L\otimes_{\mathbb{Z}}\mathbb{C})$ the set
\begin{equation}\nonumber
\mathfrak{F}_{J_{\uL}(\mathbb{Z})}:=\{(\tau,z)\in\mathfrak{H}\times(L\otimes_{\mathbb{Z}}\mathbb{C}):\tau\in\mathfrak{F}, z\in\mathscr{M}_{\uL}\}/\{\text{id}, \iota\},
\end{equation}
where $\iota$ is the reflection map $(\tau,z)\mapsto (\tau,-z)$. If $\phi$ and $\psi$ are elements of $J_{k,\uL}$ and at least one of them is a cusp form, define their {\em Petersson scalar product} as
\begin{equation}\label{eq:JPetdef}
\langle\phi,\psi\rangle:=\int_{\mathfrak{F}_{J_{\uL}(\mathbb{Z})}}\omega_{\phi,\psi}(\tau,z)dV_{\uL}(\tau,z).
\end{equation}
The integral in \eqref{eq:JPetdef} is absolutely convergent and the scalar product it defines is independent of the choice of fundamental domain.

\subsection{Vector-valued modular forms}\label{Ss:VV}

We briefly discuss the connection between Jacobi forms and vector-valued modular forms. The main reference for this subsection is \cite{Bo}.

The metaplectic group, denoted by $\Mp_2(\mathbb{Z})$, consist of elements of the form $\tilde{A}=(A,w(\tau))$, with $A=\left(\begin{smallmatrix}a&b\\c&d\end{smallmatrix}\right)$ in $\SL_2(\mathbb{Z})$ and $w(\tau)$ is a holomorphic function defined on $\mathfrak{H}$, such that $w(\tau)^2=c\tau+d$. The group law on $\Mp_2(\mathbb{Z})$ is
\begin{equation*}
(A,w(\tau))(B,v(\tau))=(AB,w(B\tau)v(\tau)).
\end{equation*}
Let $V$ be a finite dimensional vector space over $\mathbb{C}$, let $k\in\frac{1}{2}\mathbb{Z}$ and let $\rho:\Mp_2(\mathbb{Z})\to\Aut(V)$ be a finite-dimensional representation of $\Mp_2(\mathbb{Z})$, whose kernel has finite index in $\Mp_2(\mathbb{Z})$. For any function $F:\mathfrak{H}\to V$, define the Petersson slash operator
\begin{equation*}
F|_k\tilde{A}(\tau):=w(\tau)^{-2k}F(A\tau).
\end{equation*}
A holomorphic function $F:\mathfrak{H}\to V$ (write $F=(F_1,\dots,F_{\text{dim}(V)})$ for a fixed basis of $V$) that satisfies
\begin{equation*}
F|_{k}\tilde{A}(\tau)=\rho(\tilde{A})F(\tau)
\end{equation*}
for all $A\in\Mp_2(\mathbb{Z})$ and whose individual components $F_j$ ($1\leq j\leq\text{dim}(V)$) extend to holomorphic functions from $\mathfrak{H}$ to $\mathbb{C}$ is called a {\em vector-valued modular form of weight $k$ for $\rho$}. For fixed weight, denote the $\mathbb{C}$-vector space of all such functions by $M_{k}(\rho)$.

For each $x$ in $ L^\#/L$, define the Jacobi theta series
\begin{equation}\label{eq:thetadef}
\vartheta_{\uL,x}(\tau,z):=\sum_{\substack{r\in L^\#\\r\equiv x\bmod L}}e\left(\tau\beta(r)+\beta(r,z)\right).
\end{equation} 
For each $\phi$ in $ J_{k,\uL}$ with Fourier expansion as in \eqref{eq:jacobifourier}, define the following function on the upper half-plane:
\begin{equation*}
h_{\phi,x}(\tau)=\sum_{\substack{D\in\mathbb{Q}\\(D,x)\in\mathrm{supp}(\uL)}}C(D,x)q^{-D}.
\end{equation*} 
Ajouz proved in $\S2.4$ of \cite{A} that every Jacobi form has a \textit{theta expansion}:

\begin{proposition}
Each Jacobi form $\phi$ in $J_{k,\uL}$ can be written as
\begin{equation*}
\phi(\tau,z)=\sum_{x\in L^\#/L}h_{\phi,x}(\tau)\vartheta_{\uL,x}(\tau,z).
\end{equation*}

\end{proposition}
Jacobi theta series are very interesting in their own right and much can be said about them. We mention that the functions $\vartheta_{\uL,x}$ and $h_{\phi,x}$ have some modular properties with respect to $\Mp_2(\mathbb{Z})$ and we refer the reader to the PhD thesis of Boylan (see $\S3$ of \cite{Bo}), where this is discussed in detail in the more general setting of Jacobi forms over number fields. Consider the group algebra $\mathbb{C}[L^\#/L]$ of maps $L^\#/L\to\mathbb{C}$, with natural basis $\{\mathfrak{e}_x\}_{x\in L^\#/L}$. The scalar product on $\mathbb{C}[L^\#/L]$ is defined as
\begin{equation*}
\langle \sum_{x\in L^\#/L}f_x\e_x,\sum_{x\in L^\#/L}g_x\e_x\rangle=\sum_{x\in L^\#/L}f_x\overline{g_x}.
\end{equation*}

We define the {\em Weil representation associated with $\uL$} of $\Mp_2(\mathbb{Z})$ on $\Aut(\CL)$ by the following action of the generators of $\Mp_2(\mathbb{Z})$:
\begin{equation*}
\begin{split}
\rho_{\uL}(\tT)\e_x&=e(\beta(x))\e_x,\\
\rho_{\uL}(\tS)\e_x&=\frac{i^{-\frac{\rk}{2}}}{\det(\uL)^{\frac{1}{2}}}\sum_{y\in L^\#/L}e(-\beta(x,y))\e_y.
\end{split}
\end{equation*}
In general, write
\begin{equation*}
\rho_{\uL}(\tilde{A})\e_y=\sum_{x\in L^\#/L}\rho_{\uL}(\tilde{A})_{x,y}\e_x.
\end{equation*} 
It is well-known that $\rho_{\uL}$ is unitary and hence its dual representation is given by:
\begin{equation*}
\rho_{\uL}^*(\tilde{A})\e_y=\sum_{x\in L^\#/L}\overline{\rho_{\uL}(\tilde{A})_{x,y}}\e_x.
\end{equation*} 
We extend the definition of the $|_{k,\uL}$ action of $\SL_2(\mathbb{Z})$ on holomorphic functions $\phi:\mathfrak{H}\times(L\otimes\mathbb{C})\to\mathbb{C}$ to $\Mp_2(\mathbb{Z})$ in the following way: for every half-integer $k$ and for every $\tilde{A}=(A,w(\tau))$ in $\Mp_2(\mathbb{Z})$, define
\begin{equation*}
\phi|_{k,\uL}\tilde{A}(\tau,z):=\phi\left(A\tau,\frac{z}{w(\tau)^2}\right)w(\tau)^{-2k}e\left(\frac{-c\beta(z)}{w(\tau)^2}\right).
\end{equation*}
Boylan shows in $\S3.5$ of \cite{Bo} that, for every $x$ in $ L^\#/L$ and for every $\tilde{A}$ in $\Mp_2(\mathbb{Z})$, the theta series $\vartheta_{\uL,x}$ satisfies the following:
\begin{equation}\label{eq:thetamod}
\vartheta_{\uL,x}|_{\frac{\rk}{2},\uL}\tilde{A}(\tau,z)=\sum_{y\in L^\#/L}\rho_{\uL}(\tilde{A})_{x,y}\vartheta_{\uL,y}(\tau,z).
\end{equation}
The main result in $\S3.7$ of \cite{Bo} is the following theorem:

\begin{theorem}\label{T:jacobivv}
If $\uL=(L,\beta)$ is a positive-definite, even lattice, then the map
\begin{equation*}
\varphi:\phi(\tau,z)=\sum_{x\in L^\#/L}h_{\phi,x}(\tau)\vartheta_{\uL,x}(\tau,z)\mapsto\sum_{x\in L^\#/L}h_{\phi,x}(\tau)\e_x
\end{equation*}
is an isomorphism between $J_{k,\uL}$ and $M_{k-\frac{\rk}{2}}(\rho_{\uL}^*)$.
\end{theorem}

This theorem together with known results on the dimension of spaces of scalar-valued modular forms for the kernel of $\rho_{\uL}$ imply that $J_{k,\uL}=\{0\}$ if $k<\rk/2$ and that the spaces $J_{k,\uL}$ are finite dimensional. It also gives a connection between Jacobi forms of odd rank lattice index and half-integral weight modular forms and between Jacobi forms of even rank lattice index and integral weight modular forms. Note that Theorem \ref{T:jacobivv} holds over arbitrary number fields, not only over $\mathbb{Q}$.

Let $H$ be the group $\mathbb{Z}^3$ with the following composition law:
\begin{equation*}
(\lambda,\mu,t)(\lambda',\mu',t')=(\lambda+\lambda',\mu+\mu',t+t'+\lambda\mu'-\mu\lambda').
\end{equation*} 
Let $\uL$ be a positive-definite, even lattice and let $x\in L^\#/L$. The {\em Schr\"odinger representation of $H$ on $\CL$ twisted at $x$} is the representation $\sigma_x:H\to\Aut(\CL)$ given by
\begin{equation}\label{eq:schrodef}
\sigma_x(\lambda,\mu,t)\e_y:=e\left(\mu\beta(x,y)+(t-\lambda\mu)\beta(x)\right)\e_{y-\lambda x}.
\end{equation}
The Schr\"odinger representation is also unitary, which can be verified on the generators $(1,0,0)$, $(0,1,0)$ and $(0,0,1)$ of $H$:  the matrix of $\sigma_x(1,0,0)$ is a permutation matrix and the matrices of $\sigma_x(0,1,0)$ and $\sigma_x(0,0,1)$ are diagonal with diagonal entries of modulus equal to one. Define the following action of $\SL_2(\mathbb{Z})$ on H:
\begin{equation*}
\left((\lambda,\mu,t),A\right)\mapsto(\lambda,\mu,t)^A:=(\lambda a+\mu c,\lambda b+\mu d,t).
\end{equation*}
For every $\tilde{A}=(A,w(\tau))$ in $\Mp_2(\mathbb{Z})$ and every $(\lambda,\mu,t)$ in $ H$, the following relation holds between the Weil and the Schr\"odinger representations:
\begin{equation*}
\rho_{\uL}(\tilde{A})^{-1}\sigma_x(\lambda,\mu,t)\rho_{\uL}(\tilde{A})=\sigma_{x}((\lambda,\mu,t)^A).
\end{equation*}
This can be easily verified on the generators $\tT$ and $\tS$ of $\Mp_2(\mathbb{Z})$ and the generators $(1,0,0)$, $(0,1,0)$ and $(0,0,1)$ of $H$. By taking complex conjugates on both sides of this equation and re-ordering, we obtain the following relation between the dual representations:
\begin{equation}\label{eq:weilschroconj}
\sigma_x^*(\lambda,\mu,t)=\rho_{\uL}^*(\tilde{A})\sigma_{x}^*((\lambda,\mu,t)^A)\rho_{\uL}^*(\tilde{A})^{-1}.
\end{equation}

\subsection{Poincar\'e and Eisenstein series}

Let $k$ be a positive integer and let $\uL=(L,\beta)$ be a positive-definite, even lattice. Let $D$ in $\mathbb{Q}_{\leq0}$ and $r$ in $ L^{\#}$ be such that $\beta(r)\equiv D\bmod\mathbb{Z}$. Define the function
\begin{equation*}
g_{\uL,D,r}(\tau,z):=e\left(\tau\left(\beta(r)-D\right)+\beta(r,z)\right)
\end{equation*} 
on the space $\mathfrak{H}\times(L\otimes_{\mathbb{Z}}\mathbb{C})$. The following holds:

\begin{lemma}\label{L:gdr}
The function $g_{\uL,D,r}$ is invariant under the $|_{k,\uL}$ action of $J_{\uL}(\mathbb{Z})_{\infty}$. Furthermore, let $-I_2$ denote the element $\left(\left(
\begin{smallmatrix}
-1&0\\0&-1
\end{smallmatrix}
\right),(0,0)\right)$ in $J_{\uL}(\mathbb{Z})$.  The following holds:
\begin{equation*}
g_{\uL,D,r}|_{k,\uL}(-I_2)(\tau,z)=(-1)^kg_{\uL,D,-r}(\tau,z).
\end{equation*}
\end{lemma}

\begin{proof}
Let $\gamma_{\infty}=\left(\left(
\begin{smallmatrix}
1&n\\0&1
\end{smallmatrix}
\right),(0,\mu)\right)\in J_{\uL}(\mathbb{Z})_{\infty}$. Then:
\begin{equation*}
g_{\uL,D,r}|_{k,\uL}\gamma_{\infty}(\tau,z)=g_{\uL,D,r}\left(\tau+n,z+\mu\right)=g_{\uL,D,r}(\tau,z),
\end{equation*}
since $e(n(\beta(r)-D))=e(\beta(r,\mu))=1$. For the second identity, we have:
\begin{equation*}
\begin{split}
g_{\uL,D,r}|_{k,\uL}(-I_2)(\tau,z)&=(-1)^{-k}e(\tau(\beta(-r)-D)+\beta(-r,z))=(-1)^kg_{\uL,D,-r}(\tau,z).
\end{split}
\end{equation*}
\end{proof}

Let $r$ in $L^\#$ and let $D<0$ be such that $\beta(r)\equiv D\bmod\mathbb{Z}$. Define the Jacobi--Poincar\'e series of weight $k$ and index $\uL$ associated with the pair $(D,r)$ as
\begin{equation}\label{eq:Pdef}
P_{k,\uL,D,r}(\tau,z):=\sum_{\gamma\in J_{\uL}(\mathbb{Z})_{\infty}\setminus J_{\uL}(\mathbb{Z})}g_{\uL,D,r}|_{k,\uL}\gamma(\tau,z).
\end{equation}
For $r$ in $L^\#$ such that $\beta(r)\in\mathbb{Z}$, define the Jacobi--Eisenstein series of weight $k$ and index $\uL$ associated with $r$ as
\begin{equation}\label{eq:Edef}
E_{k,\uL,r}(\tau,z):=\frac{1}{2}P_{k,\uL,0,r}(\tau,z).
\end{equation}
The factor of $\frac{1}{2}$ is introduced in the definition of Eisenstein series for normalizing purposes. As a consequence of Lemma \ref{L:gdr}, the series defined in \eqref{eq:Pdef} and \eqref{eq:Edef} are independent of the choice of coset representatives of $J_{\uL}(\mathbb{Z})_{\infty}\setminus J_{\uL}(\mathbb{Z})$. The same lemma yields
\begin{equation}\label{eq:Pr-r}
P_{k,\uL,D,-r}=(-1)^kP_{k,\uL,D,r}
\end{equation}
and
\begin{equation}\label{eq:Er-r}
E_{k,\uL,-r}=(-1)^kE_{k,\uL,r}.
\end{equation}
The Eisenstein series $E_{k,\uL,r}$ was defined in $\S3.3$ of \cite{A}. It was shown there that it can be written in terms of the theta series defined in \eqref{eq:thetadef} as
\begin{equation}\label{eq:Eistheta}
E_{k,\uL,r}(\tau,z)=\frac{1}{2}\sum_{A\in \Gamma_{\infty}\setminus\Gamma}\vartheta_{\uL,r}|_{k,\uL}A(\tau,z).
\end{equation} 
It follows that $E_{k,\uL,r}$ only depends on $r\bmod L$, i.e. we can restrict their definition to $r\in\iso$. Throughout this paper, we will be working with a fixed positive-definite, even lattice $\uL$ and fixed integer weight $k$. Therefore, to ease notation, we write $g_{D,r}(\tau,z):=g_{\uL,D,r}(\tau,z)$, $P_{D,r}(\tau,z):=P_{k,\uL,D,r}(\tau,z)$ and  $E_r(\tau,z):=E_{k,\uL,r}(\tau,z)$. We will also refer to the Eisenstein series $E_{0}$ as the trivial Eisenstein series. We proceed with the main results of this paper.

\section{Jacobi--Poincar\'e series}\label{S:P}

The following holds:

\begin{theorem}\label{T:Poincare}
Let $k$ be a positive integer and let $\uL=(L,\beta)$ be a positive-definite, even lattice of rank $\rk$. The Poincar\'e series satisfies the following:
\begin{enumerate}[(i)] 
\item If $k>\rk+2$, then $P_{D,r}$ is absolutely and uniformly convergent on compact subsets of $\mathfrak{H}\times(L\otimes_{\mathbb{Z}}\mathbb{C})$ and it is an element of $S_{k,\uL}$. Furthermore, define the explicit constant
\begin{equation}\label{eq:lambda}
\lambda_{k,\uL,D}:=2^{-2k+\frac{\rk}{2}+2}\Gamma\left(k-\frac{\rk} {2}-1\right)\det(\uL)^{-\frac{1}{2}}(-\pi D)^{-k+\frac{\rk}{2}+1}.
\end{equation}
For every cusp form $\phi$ in $S_{k,\uL}$ with Fourier expansion 
\begin{equation*}
\phi(\tau,z)=\sum_{(D',r')\in\text{supp}(\uL)}C(D',r')e\left((\beta(r')-D')\tau+\beta(r',z)\right),
\end{equation*}
the following holds:
\begin{equation*}
\langle\phi,P_{D,r}\rangle=\lambda_{k,\uL,D}C(D,r).
\end{equation*}
\item The Poincar\'e series $P_{D,r}$ has the following Fourier expansion: 
\begin{equation*}
P_{D,r}(\tau,z)=\sum_{\substack{(D',r')\in\text{supp}(\uL)\\ D'<0}}G_{D,r}(D',r')e\left((\beta(r')-D')\tau+\beta(r',z)\right),
\end{equation*} 
where 
\begin{equation}\label{eq:GPoincare}
\begin{split}
G_{D,r}(D',r'):=&\delta_{L}(D,r,D',r')+(-1)^k\delta_{L}(D,-r,D',r')+\frac{2\pi i^k}{\det(\uL)^{\frac{1}{2}}}\\
&\times\left(\frac{D'}{D}\right)^{\frac{k}{2}-\frac{\rk}{4}-\frac{1}{2}}\sum_{c\geq1} J_{k-\frac{\rk}{2}-1}\left(\frac{4\pi(DD')^{\frac{1}{2}}}{c}\right)c^{-\frac{\rk}{2}-1}\\
&\times\left(H_{\uL,c}(D,r,D',r')+(-1)^kH_{\uL,c}(D,-r,D',r')\right),
\end{split}
\end{equation}
where
\begin{equation}\label{eq:deltaL}
\delta_{L}(D,r,D',r'):= 
\begin{cases} 
1, & \text{if }D'=D\text{ and }r'\equiv r\bmod L\text{ and }\\
0, &\text{otherwise},
\end{cases}
\end{equation} 
the function $J_{\alpha}$ is the $J$-Bessel function of index $\alpha$ defined in \eqref{eq:Jbessel} and $H_{\uL,c}(D,r,D',r')$ is the lattice sum
\begin{equation*}
\begin{split}
H_{\uL,c}(D,r,D',r')&:=\sum_{d(c)^{\times},\lambda(c)}e_c\big((\beta(\lambda+r)-D)d^{-1}+(\beta(r\rq{})-D\rq{})d+\beta(r',\lambda+r)\big).
\end{split}
\end{equation*}
\end{enumerate}
\end{theorem}

\begin{proof}[Proof of Theorem \ref{T:Poincare}]

It was shown in $\S2.2$ of \cite{Br2} that the Poincar\'e series of matrix index are absolutely and uniformly convergent on $\mathfrak{H}\times\mathbb{C}^{\rk}$ for $k>\rk+2$ and clearly 
the same estimates hold in the lattice index case.

For every $\delta$ in $J_{\uL}(\mathbb{Z})$,
 right multiplication by $\delta$ is an automorphism of the orbit space $J_{\uL}(\mathbb{Z})_{\infty}\setminus J_{\uL}(\mathbb{Z})$. Therefore, each $P_{D,r}$ is invariant under the $|_{k,\uL}$ action of $J_{\uL}(\mathbb{Z})$. The fact that it is an element of $S_{k,\uL}$ follows from inspecting its Fourier expansion. 

Next, we compute the Petersson scalar product of $P_{D,r}$ and an arbitrary cusp form of weight $k$ and index $\uL$. Let $\phi$ in $S_{k,\uL}$ have Fourier expansion 
\begin{equation*}
\phi(\tau,z)=\sum_{(D',r')\in\text{supp}(\uL)}C(D',r')e\left((\beta(r')-D')\tau+\beta(r',z)\right).
\end{equation*}
Using the definition of Poincar\'e series, we have:
\begin{equation}\label{eq:phiPpet}
\begin{split}
\langle\phi,P_{D,r}\rangle&=\int_{\mathfrak{F}_{J_{\uL}(\mathbb{Z})}}\omega_{\phi,P_{D,r}}(\tau,z)dV_{\uL}(\tau,z)=\sum_{\gamma\in J_{\uL}(\mathbb{Z})_{\infty}\setminus
J_{\uL}(\mathbb{Z})}\int_{\mathfrak{F}_{J_{\uL}(\mathbb{Z})}}\omega_{\phi,g_{D,r}|_{k,\uL}\gamma}(\tau,z)dV_{\uL}(\tau,z),
\end{split}
\end{equation} 
since the integrand converges absolutely and uniformly and so we can change the order of integration and summation. We want to show that
\begin{equation*}
\omega_{\phi,g_{D,r}|_{k,\uL}\gamma}(\tau,z)=\omega_{\phi,g_{D,r}}(\gamma(\tau,z))
\end{equation*}
for every $\gamma$ in $ J_{\uL}(\mathbb{Z})$, in order to use the classical unfolding argument from the theory of elliptic modular forms. Let $\gamma=(A,h)$, with $A= \left(\begin{smallmatrix} a&b\\c&d
\end{smallmatrix} \right)$ in $\SL_2(\mathbb{Z})$ and $h=(\lambda,\mu)$ in $ H_{\uL}(\mathbb{Z})$. Using \eqref{eq:omega}, we obtain:
\begin{align*}
\omega_{\phi,g_{D,r}|_{k,\uL}\gamma}(\tau,z)&=\phi(\tau,z)\overline{g_{D,r}|_{k,\uL}\gamma(\tau,z)}v^ke^{-\frac{4\pi\beta(y)}{v}}\\
&=\phi(\tau,z)\frac{(c\tau+d)^k}{|c\tau+d|^{2k}}\frac{\left|e\left(\frac{-c\beta(z+\lambda\tau+\mu)}{c\tau+d}+\tau\beta(\lambda)+\beta(\lambda,z)\right)\right|^2}{e\left(\frac{-c\beta(z+\lambda\tau+\mu)}{c\tau+d}+\tau\beta(\lambda)+\beta(\lambda,z)\right)}\overline{g_{D,r}(\gamma(\tau,z))}v^ke^{-\frac{4\pi\beta(y)}{v}}.
\end{align*} 
Since $\phi$ is invariant under the $|_{k,\uL}$ action of $J_{\uL}(\mathbb{Z})$, $\Im(A\tau)=\frac{\Im(\tau)}{|c\tau+d|^2}$ 
and $|e(z)|^2=e^{-4\pi\Im{(z)}}$
, it follows that
 \begin{align*}
\omega_{\phi,g_{D,r}|_{k,\uL}\gamma}(\tau,z)&=\phi(\gamma(\tau,z))\overline{g_{,D,r}(\gamma(\tau,z))}\Im{(A\tau)}^{k}e^{-4\pi\left(\beta(y)v^{-1}+\Im{\left(\frac{-c\beta(z+\lambda\tau+\mu)}{c\tau+d}+\tau\beta(\lambda)+\beta(\lambda,z)\right)}\right)}
\end{align*} 
and so we need to show the following:
\begin{equation*}
\begin{split}
&\beta(y)v^{-1}+\Im{\left(\frac{-c\beta(z+\lambda\tau+\mu)}{c\tau+d}+\tau\beta(\lambda)+\beta(\lambda,z)\right)}=\beta\left(\Im{\left(\frac{z+\lambda\tau+\mu}{c\tau+d}\right)}\right)\Im{(A\tau})^{-1}.
\end{split}
\end{equation*}
This follows from straight-forward calculations, using the fact that, for $z,z_1,z_2$ in $\mathbb{C}$, $\uL=(L,\beta)$ an integral lattice, $\lambda$ in $L$ and $z$ in $L\otimes_{\mathbb{Z}}\mathbb{C}$, the following equalities hold:
\begin{align*} 
&\Im(z_1z_2)=\Re(z_1)\Im(z_2)+\Im(z_1)\Re(z_2),\\
&\Im(\beta(\lambda,z))=\beta(\lambda,\Im(z)),\\
&\Re(\beta(z))=\beta(\Re(z))-\beta(\Im(z)),\\
&\Im(\beta(z))=\beta(\Re(z),\Im(z)).
\end{align*} 
Make the substitution $\gamma(\tau,z)=(\tau',z')$ in \eqref{eq:phiPpet} and, bearing in mind that the volume element $V_{\uL},(\tau,z)$ is invariant under this change of variable, we obtain:
\begin{equation*}
\langle\phi,P_{D,r}\rangle=\sum_{\gamma\in J_{\uL}(\mathbb{Z})_{\infty}\setminus
J_{\uL}(\mathbb{Z})}\int_{\gamma\mathfrak{F}_{J_{\uL}(\mathbb{Z})}}\omega_{\phi,g_{D,r}}(\tau',z')dV_{\uL}(\tau',z').
\end{equation*} 
If $\{\gamma_j\}_j$ is a set of coset representatives for $J_{\uL}(\mathbb{Z})_{\infty}\setminus
J_{\uL}(\mathbb{Z})$, then it is known that a fundamental domain for the action of $J_{\uL}(\mathbb{Z})_{\infty}$ on $\mathfrak{H}\times(L\otimes_{\mathbb{Z}}\mathbb{C})$ is given by $\mathfrak{F}_{J_{\uL}(\mathbb{Z})_{\infty}}=\cup_{j}\gamma_j\mathfrak{F}_{J_{\uL}(\mathbb{Z})}$. Thus, we obtain:
\begin{equation}\label{eq:phiPpetfinal}
\langle\phi,P_{D,r}\rangle=\int_{\mathfrak{F}_{J_{\uL}(\mathbb{Z})_{\infty}}}\phi(\tau,z)\overline{g_{D,r}(\tau,z)}v^{k}e^{-4\pi\beta(y)v^{-1}}dV_{\uL}(\tau,z).
\end{equation} 
Each $\gamma=\left(\left(\begin{smallmatrix} 1&n\\0&1 \end{smallmatrix}\right)
,(0,\mu)\right)\in J_{\uL}(\mathbb{Z})_{\infty}$ acts on $\mathfrak{H}\times(L\otimes_{\mathbb{Z}}\mathbb{C})$ via
\begin{equation}\nonumber
(\tau,z)\mapsto\left(\tau+n,z+\mu\right).
\end{equation} 
Therefore, choose the set
\begin{equation*}
\{(\tau,z)\in\mathfrak{H}\times(L\otimes\mathbb{C}):0\leq
u\leq1,v>0,x\in[0,1]^{\rk},y\in\mathbb{R}^{\rk}\}
\end{equation*} 
as a fundamental domain for the action of $J_{\uL}(\mathbb{Z})_{\infty}$ on $\mathfrak{H}\times(L\otimes_{\mathbb{Z}}\mathbb{C})$. We leave it to the reader to verify that every pair $(\tau',z')\in\mathfrak{H}\times(L\otimes\mathbb{C})$ can be written as $\gamma(\tau,z)$ for some $\gamma$ in $J_{\uL}(\mathbb{Z})_{\infty}$ and some unique $(\tau,z)$ in our chosen fundamental domain. 

Next, insert the Fourier expansion of $\phi$ in \eqref{eq:phiPpetfinal}. It is more convenient at this point to insert the classical form of the Fourier expansion of $\phi$, i.e.
\begin{equation*}
\phi(\tau,z)=\sum_{\substack{n'\in\mathbb{Z},r'\in
L^\#\\n'>\beta(r')}}c(n',r')e\left(n'\tau+\beta(r',z)\right),
\end{equation*}
where $c(n',r')=C(\beta(r')-n',r')$. Let $n:=\beta(r)-D$. Then:
\begin{align}\nonumber
\langle\phi,P_{D,r}\rangle&\nonumber=\int_0^1\int_0^\infty\int_{[0,1]^{\rk}}\int_{\mathbb{R}^{\rk}}
\sum_{\substack{n'\in\mathbb{Z},r'\in L^\#\\n'>\beta(r')}}c(n',r')e(n'\tau+\beta(r',z))\\
&\nonumber\quad\times\overline{e(n\tau+\beta(r,z))}v^{k-\rk-2}e^{-4\pi\beta(y)v^{-1}}dydxdvdu \\
&\label{eq:phiPpetF}=\sum_{\substack{n'\in\mathbb{Z},r'\in L^\#\\n'>\beta(r')}}c(n',r')\int_0^1\int_0^\infty\int_{[0,1]^{\rk}}\int_{\mathbb{R}^{\rk}}e(u(n'-n))e^{-2\pi v(n'+n)}\\
&\nonumber\quad\times e(\beta(r'-r,x)+i\beta(r+r',y))v^{k-\rk-2}e^{-4\pi\beta(y)v^{-1}}dydxdvdu\\
&\label{eq:phiPpetinner}=c(n,r)\int_0^\infty e^{-4\pi nv}v^{k-\rk-2}\int_{\mathbb{R}^{\rk}} e^{-4\pi\left(\beta(r,y)+\beta(y)v^{-1}\right)}dydv.
\end{align}
In \eqref{eq:phiPpetF}, we used the well-known formula
\begin{equation}\label{eq:exporth}
\int_0^1e(u(n'-n))du=\begin{cases}
1, & \text{if }n=n'\text{ and }\\
0, & \text{otherwise}
\end{cases}
\end{equation}
and the fact that a similar result holds for the integral in the $x$ variable: 
\begin{equation}\label{eq:exporthL}
\int_{[0,1]^{\rk}}e(\beta(r'-r,x))dx=
\begin{cases}
1, & \text{if }r=r'\text{ and }\\
0, & \text{otherwise.}
\end{cases}
\end{equation}
The inner integral in \eqref{eq:phiPpetinner} (denote it by $I$ for simplicity) can be computed by diagonalizing $\beta$. Since $\uL$ is positive-definite and symmetric, its Gram matrix $G$ can be diagonalized with a real orthogonal matrix, i.e. $G=Q^t\mathcal{D}Q$ for some $\rk\times\rk$ matrix $Q$ with real entries that satisfies $Q^tQ=I_{\rk}$ and some diagonal matrix $\mathcal{D}=\text{diag}(\alpha_1,\dots,\alpha_{\rk})$. In particular, it follows that $\prod_j\alpha_j=\det(\uL)$. Making the change of variable  $Qy=y'$, we obtain:
\begin{align*}
I&=\int_{\mathbb{R}^{\rk}}e^{-4\pi((Qr)^t\mathcal{D}(Qy)+\frac{1}{2}v^{-1}(Qy)^t\mathcal{D}(Qy))}dy=\int_{\mathbb{R}^{\rk}}e^{-4\pi(\beta'(Qr,y')+v^{-1}\beta'(y'))}dy'.
\end{align*} 
The bilinear form $\beta'(x,y):=x^t\mathcal{D}y$ is diagonalized and it satisfies $\beta'(x,y)=\beta(Q^{-1}x,Q^{-1}y)$. Writing out the exponents explicitly and dropping the primes yields
\begin{align*}
I&=\int_{\mathbb{R}^{\rk}}e^{-4\pi\left(\sum_{j=1}^{\rk}\alpha_j(Qr)_jy_j+(2v)^{-1}\sum_{j=1}^{\rk}\alpha_jy_j^2\right)}dy=\prod_{j=1}^{\rk}\left(\int_{\mathbb{R}}e^{-2\pi\alpha_j(2(Qr)_jy_j+v^{-1}y_j^2)}dy_j\right).
\end{align*} 
Complete the square in the exponent and obtain
\begin{align*}
I&=\prod_{j=1}^{\rk}\left(e^{2\pi\alpha_jv(Qr)_j^2}\int_{\mathbb{R}}e^{-2\pi\alpha_jv^{-1}(y_j+v(Qr)_j)^2}dy_j\right)=\prod_{j=1}^{\rk}\left(e^{2\pi\alpha_jv(Qr)_j^2}\int_{\mathbb{R}}e^{-2\pi\alpha_jv^{-1}y_j^2}dy_j\right).
\end{align*} 
Next, substitute $(2\pi\alpha_jv^{-1})^{\frac{1}{2}}y_j=x_j$ and use the standard Gaussian integral to obtain
\begin{align*}
I=e^{4\pi v\frac{1}{2}\sum_{j=1}^{\rk}\alpha_j(Qr)_j^2}\prod_{j=1}^{\rk}\left(\frac{v}{2\alpha_j}\right)^{\frac{1}{2}}=e^{4\pi v\beta(r)}v^{\frac{\rk}{2}}2^{-\frac{\rk}{2}}\det(\uL)^{-\frac{1}{2}}. 
\end{align*} 
Thus, 
\begin{align*}
\langle\phi,P_{D,r}\rangle&=c(n,r)2^{-\frac{\rk}{2}}\det(\uL)^{-\frac{1}{2}}\int_0^\infty
e^{-4\pi v(n-\beta(r))}v^{k-\frac{\rk}{2}-2}dv=\lambda_{k,\uL,D}C(D,r), 
\end{align*}
where the constant $\lambda_{k,\uL,D}$ is defined in \eqref{eq:lambda}. Hence, the proof of item $(i)$ is complete.

We proceed with the proof of $(ii)$. To obtain the Fourier expansion of $P_{D,r}$, choose
\begin{equation*}
\left\{\left(A,(\lambda,0)^A\right):A\in\SL_2(\mathbb{Z})_{\infty}\setminus\SL_2(\mathbb{Z}),\lambda\in
L\right\}
\end{equation*} 
as a set of coset representatives for $J_{\uL}(\mathbb{Z})_{\infty}\setminus J_{\uL}(\mathbb{Z})$. Coset representatives of $\SL_2(\mathbb{Z})_{\infty}\setminus\SL_2(\mathbb{Z})$ are well-known and given by matrices $A= \left(\begin{smallmatrix} a&b\\c&d\end{smallmatrix}\right)$ with $\gcd(c,d)=1$ and, for each pair $(c,d)$, choose $a$ and $b$ in $\mathbb{Z}$ such that $ad-bc=1$. The vectors $(\lambda,0)^A$ are therefore equal to $(a\lambda,b\lambda)$, with $\lambda$ in $ L$. Inserting this into \eqref{eq:Pdef} and letting $n:=\beta(r)-D$ as usual, we obtain:
\begin{equation*}
\begin{split}
P_{D,r}(\tau,z)&=\sum_{\substack{\lambda\in L,(c,d)\in\mathbb{Z}^2\\ \gcd(c,d)=1}}e\left(\frac{-c\beta(z+(a\tau+b)\lambda)}{c\tau+d}+a^2\tau\beta(\lambda)+a\beta(\lambda,z)\right)\\
&\quad\times
e\left(n\frac{a\tau+b}{c\tau+d}+\beta\left(r,\frac{z+(a\tau+b)\lambda}{c\tau+d}\right)\right)(c\tau+d)^{-k}\\
&=\sum_{\substack{(c,d)\in\mathbb{Z}^2\\ \gcd(c,d)=1}}\sum_{\lambda\in L}(c\tau+d)^{-k}e\bigg(\beta(z)\frac{-c}{c\tau+d}+\beta(\lambda)\frac{a\tau+b}{c\tau+d}\\
&\quad+\frac{\beta(\lambda,z)}{c\tau+d}+n\frac{a\tau+b}{c\tau+d}+\frac{\beta(r,z)}{c\tau+d}+\beta(r,\lambda)\frac{a\tau+b}{c\tau+d}\bigg),
\end{split}
\end{equation*} 
after rearranging terms. In order to obtain the factor of $\beta(\lambda)$, we wrote
\begin{equation*}
\begin{split}
a^2\tau(c\tau+d)-c(a\tau+b)^2=a\tau-abc\tau-b(ad-1)=(a\tau+b)-ab(c\tau+d)
\end{split}
\end{equation*}
and, to obtain the factor of $\beta(\lambda,z)$, we used the following well-known identity
\begin{equation}\label{eq:modulartrick}
\frac{a\tau+b}{c\tau+d}=\frac{a}{c}-\frac{1}{c(c\tau+d)}.
\end{equation}
We can split this sum up into two sums, according to whether $c=0$ or $c\neq0$. When $c=0$, we have $d=\pm1$ and this forces $a=d=\pm1$ and $b$ can be any integer. Since $r\in L^{\#}$ and $\lambda\in L$, we have $\beta(r,\lambda),\beta(\lambda)\in\mathbb{Z}$ and we obtain a contribution of 
\begin{align} &\nonumber\sum_{\lambda\in
L}\Big[e\left(\beta(\lambda)\tau+\beta(\lambda,z)+n\tau+\beta(r,z)+\beta(r,\lambda)\tau\right)\\
&\nonumber+(-1)^{-k}e(\beta(\lambda)\tau-\beta(\lambda,z)+n\tau-\beta(r,z)+\beta(r,\lambda)\tau)\Big]\\
=&\label{eq:0contrib}\sum_{\lambda\in
L}e\left((\beta(\lambda+r)-D)\tau\right)\left[e(\beta(\lambda+r,z))+(-1)^ke(\beta(-(\lambda+r),z))\right].
\end{align} 
We want to write this expression as a standard Fourier expansion of a Jacobi form (like the one in \eqref{eq:jacobifourier}). Set $r':=\lambda+r$ in \eqref{eq:0contrib}, which implies we are summing over all $r'$ in $ L^\#$ such that $r'\equiv r\bmod L$. Introduce the summation over all $D'$ in $\mathbb{Q}$ with $D'<0$ such that $(D',r')\in\text{supp}(\uL)$ and impose the condition that $D'=D$. Equation \eqref{eq:0contrib} becomes
\begin{equation*}
\sum_{\substack{(D',r')\in\text{supp}(\uL)\\D'<0}}e\left((\beta(r')-D')\tau+\beta(r',z)\right)\left[\delta_{L}(D,r,D',r')+(-1)^k\delta_{L}(D,r,D',-r')\right]
\end{equation*} 
where $\delta_{L}$ is defined in \eqref{eq:deltaL}.

For the contribution coming from terms with $c\neq0$, Lemma \ref{L:gdr} implies that the terms with $c<0$ are obtained from those with $c>0$, by multiplying their contribution with $(-1)^k$ and replacing $z$ by $-z$. Thus, we focus on the former case. Use \eqref{eq:modulartrick} again to write the contribution coming from terms with $c>0$ as
\begin{align*} 
&\sum_{\substack{c>0\\ \gcd(c,d)=1}}\sum_{\lambda\in
L}(c\tau+d)^{-k}e\bigg(-\frac{c}{c\tau+d}\beta\left(z-\frac{1}{c}\lambda\right)+\frac{a\beta(\lambda)}{c}\\
&+\beta\left(r,\frac{1}{c\tau+d}\left(z-\frac{1}{c}\lambda\right)+\frac{a}{c}\lambda\right)+n\left(\frac{a}{c}-\frac{1}{c(c\tau+d)}\right)\bigg).
\end{align*} 
Write $d$ as $d'+\alpha c$, where $d'$ is the reduction of $d$ modulo $c$ and $\alpha\in\mathbb{Z}$.  As $d$ runs through $\mathbb{Z}$ with the condition that $\gcd(d,c)=1$ in the above equation, the new variable $d'$ runs through congruence classes modulo $c$ that are coprime to $c$ (we will drop the prime from the notation and write $d(c)^\times$ for simplicity) and $\alpha$ runs through $\mathbb{Z}$. Similarly, write $\lambda$ as $\lambda'+\mu c$, where $\lambda'$ is the reduction of $\lambda$ modulo $cL$ and $\mu\in L$. It is clear that $\lambda'$ runs through the coset representatives of  $L/cL$ and that $\mu$ runs through $L$. We obtain a contribution of
\begin{align*}
&\sum_{\substack{c>0,\alpha\in\mathbb{Z},d(c)^{\times}\\ \mu\in L,\lambda(c)}}c^{-k}\left(\tau+\frac{d}{c}+\alpha\right)^{-k}e\Bigg(\frac{-\beta\left(z-\frac{1}{c}\lambda-\mu\right)}{\tau+\frac{d}{c}+\alpha}+\frac{a\beta(\lambda)}{c}\\
&\quad\quad+\frac{\beta\left(r,z-\frac{1}{c}\lambda-\mu\right)}{c(\tau+\frac{d}{c}+\alpha)}+\frac{a\beta(r,\lambda)}{c}+n\left(\frac{a}{c}-\frac{1}{c^2(\tau+\frac{d}{c}+\alpha)}\right)\Bigg)\\
&\quad\quad=\sum_{c>0}c^{-k}\sum_{d(c)^{\times},\lambda(c)}e_c\left(\left(\beta(\lambda)+\beta(r,\lambda)+n\right)d^{-1}\right)\mathscr{F}_{c;(n,r)}\left(\tau+\frac{d}{c},z-\frac{1}{c}\lambda\right),
\end{align*} 
where $d^{-1}$ is the inverse of $d$ modulo $c$ and we have used the fact that $ad\equiv1\bmod c$. Furthermore, the function $\mathscr{F}_{c;(n,r)}:\mathfrak{H}\times(L\otimes_{\mathbb{Z}}\mathbb{C})\to\mathbb{C}$ is defined as
\begin{align*}
\mathscr{F}_{c;(n,r)}(\tau,z):=&\sum_{\alpha\in\mathbb{Z},\mu\in L}(\tau+\alpha)^{-k}e\left(\frac{-\beta(z-\mu)}{\tau+\alpha}+\frac{\beta(r,z-\mu)}{c(\tau+\alpha)}-\frac{n}{c^2(\tau+\alpha)}\right). 
\end{align*}
This function has period $1$ in $\tau$ and period $L$ in $z$ and hence it has a Fourier expansion of the form
\begin{equation*}
\mathscr{F}_{c;(n,r)}(\tau,z)=\sum_{n'\in\mathbb{Z},r'\in L^\#}f(n',r')e\left(n'\tau+\beta(r',z)\right). 
\end{equation*}
We can compute the Fourier coefficients of $\mathscr{F}_{c;(n,r)}$ by integrating it against an appropriate exponential function, since $P_{D,r}$ is absolutely and uniformly convergent in $\tau$ and $z$, and hence so is $\mathscr{F}$. We remind the reader that we write $\tau=u+iv$ and $z=x+iy$. For fixed $v>0$, $y$ in $\mathbb{C}^{\rk}$, $m$ in $\mathbb{Z}$ and $s$ in $L^{\#}$, we have:
\begin{equation*}
 \begin{split}
  &\int_{[0,1]}\int_{[0,1]^{\rk}}\mathscr{F}_{c;(n,r)}(\tau,z)e\left(-mu-\beta(s,x)\right)dxdu=f(m,s)e(imv)e\left(\beta(s,iy)\right),
 \end{split}
\end{equation*}
by the standard orthogonality relations \eqref{eq:exporth} and \eqref{eq:exporthL}. Thus, we can evaluate $f(n',r')$ as
\begin{equation}\label{eq:Ffourier}
\begin{split}
f(n',r')&=\sum_{\alpha\in\mathbb{Z},\mu\in
L}\int_{[0,1]}\int_{[0,1]^{\rk}}(\tau+\alpha)^{-k}e\bigg(\frac{-1}{\tau+\alpha}\beta(z-\mu)\\
&+\frac{1}{c(\tau+\alpha)}\beta(r,z-\mu)-\frac{n}{c^2(\tau+\alpha)}\bigg)e(-n'\tau)e\left(-\beta(r',z)\right)dxdu\\
&=\int_{-\infty}^{\infty}\tau^{-k}e(-n'\tau)\int_{-\infty}^{\infty}\dots\int_{-\infty}^{\infty} e\left(\frac{-1}{\tau}\beta(z)+\frac{1}{c\tau}\beta(r,z)-\frac{n}{c^2\tau}-\beta(r',z)\right)dxdu.
\end{split}
\end{equation} 
Making the change of variable $z\mapsto z+\frac{1}{c}r-\tau r'$ and setting $D'=\beta(r')-n'$ implies the inner multiple integral becomes
\begin{align*}
e_c\left(-\beta(r',r)\right)e\bigg(\beta(r')\tau+\frac{D}{c^2\tau}\bigg)\int_{-\infty}^{\infty}\dots\int_{-\infty}^{\infty}e\left(\frac{-\beta(z)}{\tau}\right)dx.
\end{align*} 
Using the generalized Gaussian integral, it can be shown that
\begin{equation}\label{eq:innermult}
\int_{-\infty}^{\infty}\dots\int_{-\infty}^{\infty}e\left(\frac{-\beta(z)}{\tau}\right)dx=\det(\uL)^{-\frac{1}{2}}\left(\frac{\tau}{i}\right)^{\frac{\rk}{2}}
\end{equation}
and hence 
\begin{equation*}
f(n',r')=\frac{e_c\left(-\beta(r',r)\right)}{\det(\uL)^{\frac{1}{2}}}\int_{-\infty}^{\infty}\left(\frac{\tau}{i}\right)^{\frac{\rk}{2}}\tau^{-k}e\left(D'\tau+\frac{D}{c^2\tau}\right)du.
\end{equation*}
To compute this integral, consider the two separate cases: $D'\geq0$ and $D'<0$. 

If $D'\geq0$, then let $R>0$ and consider the closed contour integral
\begin{equation}\label{eq:intclosed1}
 \oint\frac{1}{(u+iv)^{k-\frac{\rk}{2}}}e\left(D'(u+iv)+\frac{D}{c^2(u+iv)}\right)du,
\end{equation}
over the contour in Figure \ref{F:contour}, formed by traversing the line segment $L=\{t:-R\leq t\leq R\}$ from left to right and the semi-circle $\mathcal{C}=\{Re^{i\theta}:0\leq\theta\leq\pi\}$ in the counter-clockwise direction. 
\begin{figure}[ht]\label{F:contour}
\centering 
\begin{tikzpicture}[scale=2]
    \coordinate (XAxisMin) at (-1.5,0);
    \coordinate (XAxisMax) at (1.5,0);
    \coordinate (YAxisMin) at (0,-0.5);
    \coordinate (YAxisMax) at (0,1.5);

    \draw (1,0) node [below] {$R$};
    \draw (-1.05,0) node [below] {$-R$};
    \draw [thin, black,-latex] (XAxisMin) -- (XAxisMax);
    \draw [thin, black,-latex] (YAxisMin) -- (YAxisMax);
    
\draw[ultra thick,-->, cadmiumred] (1,0) arc (0:180:1);
\draw (-0.65,0.8) node [left] {$\mathcal{C}$};

\draw[ultra thick,cadmiumred,-->](-1,0) -- (1,0);

\filldraw[black] (0,0) circle (1pt);
\draw (0,0.1) node [right] {$O$};

\draw (0.5,0) node [above] {$L$};

   \end{tikzpicture}
\caption{}
\end{figure}

The integral we seek is
\begin{equation}\label{eq:intC1}
\lim_{R\to\infty} \int_{L}\frac{1}{(u+iv)^{k-\frac{\rk}{2}}}e\left(D'(u+iv)+\frac{D}{c^2(u+iv)}\right)du.
\end{equation}
The integrand is holomorphic inside our chosen contour and therefore \eqref{eq:intclosed1} is equal to zero by Cauchy's Theorem. Using the estimation lemma from complex analysis and our chosen parametrization, the absolute value of the integral over the contour $\mathcal{C}$ is less than or equal to
\begin{equation*}
  \pi R\max_{u\in\mathcal{C}}\left|\frac{1}{(u+iv)^{k-\frac{\rk}{2}}}e\left(D'(u+iv)+\frac{D}{c^2(u+iv)}\right)\right|.
\end{equation*}
It is straight forward to show that this expression converges to zero as $R\to\infty$. Therefore, so does \eqref{eq:intC1} and hence the Fourier coefficients $f(n',r')$ vanish when $D'\geq0$.

If $D'<0$, then make the substitution $\tau=\frac{i}{c}(\frac{D}{D'})^{1/2}s$ and write $s=u+iv$ by abuse of notation. Integrating in $\Re(\tau)$ from $-\infty$ to $\infty$ means we are integrating in $\Im(s)$ from $\infty$ to $-\infty$ and note that $\Re(s)>0$. We obtain: 
\begin{align}\nonumber
f(n',r')=&\det(\uL)^{-\frac{1}{2}}e_c\left(-\beta(r',r)\right)i^{-k}c^{k-\frac{\rk}{2}-1}\left(\frac{D'}{D}\right)^{\frac{k}{2}-\frac{\rk}{4}-\frac{1}{2}}\\
&\label{eq:intlapl}\times\int_{-\infty}^{\infty}s^{\frac{\rk}{2}-k}\exp\left(\frac{2\pi(DD')^{\frac{1}{2}}}{c}\left(s-s^{-1}\right)\right)dv.
\end{align} 
For fixed $\sigma>0$ and $\kappa>0$, the functions 
\begin{equation*}
t\mapsto f(t)=\left(\frac{t}{\kappa}\right)^{\frac{\sigma-1}{2}}J_{\sigma-1}(2\sqrt{\kappa t})\text{ ($t>0$)}
\end{equation*}
and 
\begin{equation*}
s\mapsto F(s)=s^{-\sigma}e^{-\frac{\kappa}{s}}\text{ ($\Re(s)>0$)}
\end{equation*}
are mutually inverse with respect to the Laplace transform, i.e.
\begin{equation*}
f(t)=\frac{1}{2\pi i}\int_{C-i\infty}^{C+i\infty}F(s)e^{st}ds, \quad\quad C>0.
\end{equation*}
Thus, if we take $\kappa=\frac{2\pi(DD')^{1/2}}{c}$ and $\sigma=k-\frac{\rk}{2}$, then the integral \eqref{eq:intlapl} is equal to $2\pi\cdot f\left(t\right)$, where $t=\kappa=\frac{2\pi(DD')^{1/2}}{c}$. Therefore, 
\begin{align*}
f(n',r')&=\frac{2\pi i^{-k}}{\det(\uL)^{\frac{1}{2}}}e_c\left(-\beta(r',r)\right)c^{k-\frac{\rk}{2}-1}\left(\frac{D'}{D}\right)^{\frac{k}{2}-\frac{\rk}{4}-\frac{1}{2}}J_{k-\frac{\rk}{2}-1}\left(\frac{4\pi(DD')^{\frac{1}{2}}}{c}\right)
\end{align*}
and we obtain the following contribution from the terms with $c>0$:
\begin{equation*}
\begin{split}
&\sum_{c>0}\sum_{d(c)^{\times},\lambda(c)}e_c\left(\left(\beta(\lambda)+\beta(r,\lambda)+n\right)d^{-1}-\beta(r',r)\right)\sum_{\substack{n'\in\mathbb{Z},r'\in L^\#\\\beta(r')<n'}}
\frac{2\pi i^{-k}}{\det(\uL)^{\frac{1}{2}}}c^{-\frac{\rk}{2}-1}\\
&\times\left(\frac{D'}{D}\right)^{\frac{k}{2}-\frac{\rk}{4}-\frac{1}{2}}J_{k-\frac{\rk}{2}-1}\left(\frac{4\pi(DD')^{\frac{1}{2}}}{c}\right)e\left(n'\left(\tau+\frac{d}{c}\right)+\beta\left(r',z-\frac{\lambda}{c}\right)\right)\\
&=\frac{2\pi
i^k}{\det(\uL)^{\frac{1}{2}}}\left(\frac{D'}{D}\right)^{\frac{k}{2}-\frac{\rk}{4}-\frac{1}{2}}\sum_{(D',r')\in\text{supp}(\uL)}\sum_{c>0} J_{k-\frac{\rk}{2}-1}\left(\frac{4\pi(DD')^{\frac{1}{2}}}{c}\right)\\
&\times c^{-\frac{\rk}{2}-1}(-1)^k H_{\uL,c}(D,r,D',-r')e\left((\beta(r')-D')\tau+\beta(r',z)\right),
\end{split}
\end{equation*}
where
\begin{equation*}
\begin{split}
H_{\uL,c}(D,r,D',r')&:=\sum_{d(c)^{\times},\lambda(c)}e_c\big((\beta(\lambda+r)-D)d^{-1}+(\beta(r\rq{})-D\rq{})d+\beta(r',\lambda+r)\big).
\end{split}
\end{equation*}
We remind the reader that, in the last equation, $d$ runs through $\mathbb{Z}_c^{\times}$, the set of invertible residue classes modulo $c$, and $\lambda$ runs through a complete set of representatives of $L/cL$ and $d^{-1}$ denotes the inverse of $d$ modulo $c$.

Furthermore, when $c<0$ we obtain the same contribution multiplied by $(-1)^k$ and with $z$ replaced by $(-z)$. Using the bi-linearity of $\beta$, we can simply move the
minus sign in front of $r'$ and relabel $r':=(-r')$ by abuse of notation. The observation that $H_{\uL,c}(D,r,D',-r')=H_{\uL,c}(D,-r,D',r')$ concludes the proof. 
\end{proof}

Note that \eqref{eq:Pr-r} follows from Theorem \ref{T:Poincare}, $(ii)$. Furthermore, it is clear from the definitions of $\delta_L(D,r,D',r')$ and $H_{\uL,c}(D,r,D',r')$ that the Poincar\'e series $P_{D,r}$ only depends on $r\bmod L$. We can rewrite the lattice sum $H_{\uL,c}(D,r,D',r')$ with the help of Kloosterman sums. For fixed $c$ in $\mathbb{N}$ and $m,n$ in $\mathbb{Z}$, the Kloosterman sum $K(m,n;c)$ is defined as
\begin{equation*}
K(m,n;c)\sum_{d\in\mathbb{Z}_c^{\times}}e_c(md+nd^{-1}),
\end{equation*}
where $d^{-1}$ denotes the inverse of $d$ modulo $c$, as usual. Thus,
\begin{equation*}
H_{\uL,c}(D,r,D',r')=\sum_{\lambda(c)}e_c(\beta(r',\lambda+r))K(\beta(r')-D',\beta(\lambda+r)-D;c).
\end{equation*}
Finally, let $m$ be a positive-definite, symmetric, half-integral $g\times g$ matrix. When applied to the lattice $\uL=(\mathbb{Z}^{g,1},(x,y)\mapsto x^t2my)$, Theorem \ref{T:Poincare} agrees with the results in \cite{BK}. With the notation in \cite{BK}, let $\mathcal{D}$ and $\mathcal{D}'$ be the determinants of the block matrices $\left(\begin{smallmatrix}2n&r\\r^t&2m
\end{smallmatrix}\right)$ and $\left(\begin{smallmatrix}2n'&r'\\r'^t&2m
\end{smallmatrix}\right)$, respectively. Then:
 \begin{equation*}
 \begin{split}
 \lambda_{k,m,\mathcal{D}}&=\lambda_{k,\uL,-\frac{\mathcal{D}}{2\det(\uL)}},\\
g_{k,m;(n,r)}(n',r')&=G_{-\frac{\mathcal{D}}{2\det(\uL)},\frac{1}{2}m^{-1}r^t}\left(-\frac{\mathcal{D'}}{2\det(\uL)},\frac{1}{2}m^{-1}r'^t\right),
\end{split}
\end{equation*} 
where the left-hand side uses the notation in \cite{BK}. 

\section{Jacobi--Eisenstein series}\label{S:E}

The analogue result of Theorem \ref{T:Poincare} for Eisenstein series is the following:

\begin{theorem}\label{T:Eisenstein}
Let $k$ be an integer and let $\uL=(L,\beta)$ be a positive-definite, even lattice of rank $\rk$. The Eisenstein series satisfies the following:
\begin{enumerate}[(i)] 
\item If $k>\frac{\rk}{2}+2$, then $E_{r}$ is absolutely and uniformly convergent on compact subsets of $\mathfrak{H}\times(L\otimes_{\mathbb{Z}}\mathbb{C})$ and it is an element of $J_{k,\uL}$. Furthermore, it is orthogonal to cusp forms of the same weight and index. 
\item The Eisenstein series $E_{r}$ has the following Fourier expansion: 
\begin{equation*}
\begin{split}
E_{r}(\tau,z)&=\frac{1}{2}\sum_{\substack{r'\in L^\#\\
\beta(r')\in\mathbb{Z}}}\left(\delta(r,r')+(-1)^{k}\delta(-r,r')\right)e\left(\tau\beta(r')+\beta(r',z)\right)\\
&+\sum_{\substack{(D',r')\in\text{supp}(\uL)\\ D'<0}}G_{r}(D',r')e\left((\beta(r')-D')\tau+\beta(r',z)\right),
\end{split}
\end{equation*} 
where 
\begin{equation*}
 \delta(r,r'):=
 \begin{cases}
                1, & \text{if }r'\equiv r\bmod L\text{ and }\\
                0, & \text{otherwise}
 \end{cases}
\end{equation*}
and
\begin{equation}\label{eq:GEisenstein}
\begin{split}
G_{r}(D',r')&:=\frac{(2\pi)^{k-\frac{\rk}{2}}i^k}{2\det(\uL)^{\frac{1}{2}}\Gamma\left(k-\frac{\rk}{2}\right)}(-D')^{k-\frac{\rk}{2}-1}\\
&\times\sum_{c\geq1}c^{-k}\left(H_{\uL,c}(r,D',r')+(-1)^kH_{\uL,c}(-r,D',r')\right),
\end{split}
\end{equation}
where $H_{\uL,c}(r,D',r')$ is the lattice sum
\begin{equation}\label{eq:Heis}
\begin{split}
H_{\uL,c}(r,D',r'):=\sum_{\lambda(c),d(c)^\times}&e_c\left(\beta(\lambda+r)d^{-1}+(\beta(r\rq{})-D\rq{})d+\beta(r\rq{},\lambda+r)\right).
\end{split}
\end{equation}
\end{enumerate}
\end{theorem}

This theorem can be proved by following the steps in the proofs of Theorem \ref{T:Poincare}, up to a certain point. We pick up from where the differences arise.

\begin{proof}[Proof of Theorem \ref{T:Eisenstein}]

It was stated in \cite{A} that the series defined in \eqref{eq:Edef} converges absolutely and uniformly on compact subsets of $\mathfrak{H}\times(L\otimes_{\mathbb{Z}}\mathbb{C})$ for $k>\frac{\rk}{2}+2$. It was also shown that it is independent of the choice of coset representatives of $J_{\uL}(\mathbb{Z})_{\infty}\setminus J_{\uL(\mathbb{Z})}$ and it is invariant under the $|_{k,\uL}$ action of $J_{k,\uL}$. The fact that it is an element of $J_{k,\uL}$ follows from inspecting its Fourier expansion

We can compute the Petersson scalar product of $E_{r}$ and an arbitrary cusp form $\phi$ of weight $k$ and index $\uL$ in the same way as in the proof of Theorem \ref{T:Poincare}, item $(i)$, up until \eqref{eq:phiPpetF}. At this point, due to the orthogonality relations \eqref{eq:exporthL}, the integral in $x$ vanishes. This is due to the fact that $r'$ cannot be equal to $r$, since $\phi$ is a cusp form and therefore has no terms in its Fourier expansion with $\beta(r')$ in $\mathbb{Z}$.

Proceeding to $(ii)$, when we analyse the contribution coming from terms with $c=0$ in \eqref{eq:0contrib}, we set $r':=\lambda+r$ in \eqref{eq:0contrib} as before, which implies that we are summing over all $r'$ in $ L^\#$ such that $r'\equiv r\bmod L$. Since $D=0$ in this case, the contribution is
\begin{equation*}
 \sum_{\substack{r'\in L^\#\\ r'\equiv r\bmod L}}e\left(\beta(r')\tau\right)\left(e\left(\beta(r',z)\right)+(-1)^ke\left(\beta(-r',z)\right)\right).
\end{equation*}
Thus, we obtain the desired singular term in the Fourier expansion of $E_{r}$.

In the contribution coming from terms with $c\neq0$, the change arises in the Fourier coefficients \eqref{eq:Ffourier} of $\mathscr{F}_{c;(n,r)}$. They are now equal to
\begin{equation*}
\begin{split}
f(n',r')&=\det(\uL)^{-\frac{1}{2}}e_c\left(-\beta(r',r)\right)\int_{-\infty}^{\infty}\left(\frac{\tau}{i}\right)^{\frac{\rk}{2}}\tau^{-k}e\left(D'\tau\right)du.
\end{split}
\end{equation*}
If $D'\geq0$, then applying the same estimates as before with $D=0$ yields $f(n',r')=0$. When $D'<0$, we need to compute the integral
\begin{equation*}
I=\int_{-\infty+iv}^{\infty+iv}\left(\frac{\tau}{i}\right)^{\frac{\rk}{2}}\tau^{-k}e\left(D'\tau\right)d\tau.
\end{equation*}
Make the substitution $2\pi i D'\tau=z$ to obtain 
\begin{equation*}
I=i^{-\frac{\rk}{2}}(2\pi i D')^{k-\frac{\rk}{2}}\int_{C+i\infty}^{C-i\infty}z^{-\left(k-\frac{\rk}{2}\right)}e^zdz,
\end{equation*}
where $C:=-D'v$ is a positive constant. For fixed $\nu>0$, the functions 
\begin{equation*}
t\mapsto f(t)=t^{\nu-1}
\end{equation*}
and 
\begin{equation*}
s\mapsto F(s)=\Gamma(\nu)s^{-\nu}
\end{equation*}
are mutually inverse with respect to the Laplace transform. 
Taking $\nu=k-\frac{\rk}{2}$ yields
\begin{equation*}
I=i^{-\frac{\rk}{2}}(2\pi i D')^{k-\frac{\rk}{2}}\frac{(-2\pi i)f(1)}{\Gamma\left(k-\frac{\rk}{2}\right)}=\frac{(2\pi)^{k-\frac{\rk}{2}}i^{-k}}{\Gamma\left(k-\frac{\rk}{2}\right)}(-D')^{k-\frac{\rk}{2}-1}
\end{equation*}
and hence
\begin{equation*}
 f(n',r')=\det(\uL)^{-\frac{1}{2}}e_c\left(-\beta(r',r)\right)\frac{(2\pi)^{k-\frac{\rk}{2}}i^{-k}}{\Gamma\left(k-\frac{\rk}{2}\right)}(-D')^{k-\frac{\rk}{2}-1}.
\end{equation*}
Thus, we obtain the following contribution from the terms with $c>0$:
\begin{align*}
&\sum_{\substack{n'\in\mathbb{Z},r'\in L^\#\\\beta(r')<n'}}\sum_{c>0}c^{-k}\sum_{d(c)^{\times},\lambda(c)}e_c\left(\left(\beta(\lambda)+\beta(r,\lambda)+\beta(r)\right)d^{-1}-\beta(r',r)\right)\\
&\times\frac{(2\pi)^{k-\frac{\rk}{2}}i^{-k}}{\det(\uL)^{\frac{1}{2}} \Gamma\left(k-\frac{\rk}{2}\right)}(-D')^{k-\frac{\rk}{2}-1}e\left(n'\left(\tau+\frac{d}{c}\right)+\beta\left(r',z-\frac{\lambda}{c}\right)\right)\\
&=\sum_{(D',r')\in\text{supp}(\uL)}\frac{(2\pi)^{k-\frac{\rk}{2}}i^{k}}{\det(\uL)^{\frac{1}{2}}\Gamma\left(k-\frac{\rk}{2}\right)}(-D')^{k-\frac{\rk}{2}-1}\\
&\times\sum_{c>0}c^{-k}(-1)^kH_{\uL,c}(-r,D',r')e\left((\beta(r')-D')\tau+\beta(r',z)\right),
\end{align*} 
where
\begin{equation*}
\begin{split}
H_{\uL,c}(r,D',r'):=\sum_{\lambda(c),d(c)^\times}&e_c\left(\beta(\lambda+r)d^{-1}+(\beta(r\rq{})-D\rq{})d+\beta(r\rq{},\lambda+r)\right).
\end{split}
\end{equation*}
We remind the reader that $d$ runs through $\mathbb{Z}_c^{\times}$, $\lambda$ runs through a complete set of representatives of $L/cL$ and $d^{-1}$ denotes the inverse of $d$ modulo $c$.

Furthermore, when $c<0$, we obtain the same contribution, multiplied by $(-1)^k$ and with $z$ replaced by $(-z)$. Using the bi-linearity of $\beta$, we can again move the minus sign in front of $r'$ and relabel $r':=(-r')$ by abuse of notation, obtaining the desired result and completing the proof.

\end{proof}

\begin{remark}
For $k>\rk+2$, the ``non-singular'' Fourier coefficients of the Eisenstein series can be obtained from those of the Poincar\'e series by using \eqref{eq:Jbessel}. The $J$-Bessel function $J_{\alpha}$ has the following well-known asymptotic form for $0<x\ll (\alpha+1)^{\frac{1}{2}}$:
\begin{equation*}
J_{\alpha}(x)\sim\frac{1}{\Gamma(\alpha+1)}\left(\frac{x}{2}\right)^{\alpha}.
\end{equation*}
Therefore, if we view $D$ as a parameter in $\mathbb{R}$ and take the limit as $D\to0$ in \eqref{eq:GPoincare}, we obtain:
\begin{equation*}
\begin{split}
\lim_{D\to0}G_{D,r}(D',r')&=\frac{i^k(2\pi)^{k-\frac{\rk}{2}}(-D')^{k-\frac{\rk}{2}-1}}{\det(\uL)^{\frac{1}{2}}\Gamma\left(k-\frac{\rk}{2}\right)}\\
&\times\sum_{c\geq1}c^{-k}\left(H_{\uL,c}(0,r,D',r')+(-1)^kH_{\uL,c}(0,-r,D',r')\right)
\end{split}
\end{equation*}
and clearly $H_{\uL,c}(r,D',r')=H_{\uL,c}(0,r,D',r')$.
\end{remark}

\section{Trivial Eisenstein series}\label{S:E0}

In the following section we will show that, when $k$ is even, the Fourier coefficients of $E_{r}$ can be written as finite linear combinations of Fourier coefficients of the trivial Eisenstein series $E_{0}$ for every $r$ in $\iso$. In this section, we give an explicit formula for the latter:

\begin{theorem}\label{T:E0}
The Eisenstein series $E_{0}$ vanishes identically when $k$ is odd. When $k$ is even, it has the following Fourier expansion: 
\begin{equation*}
\begin{split}
E_{0}(\tau,z)&=\vartheta_{\uL,0}(\tau,z)+\sum_{\substack{(D,x)\in\text{supp}(\uL)\\ D<0}}G_{0}(D,x)e\left((\beta(x)-D)\tau+\beta(x,z)\right),
\end{split}
\end{equation*} 
where $\vartheta_{\uL,0}$ is a theta series as in \eqref{eq:thetadef}. When $\rk$ is even, write  $\Delta(\uL)=\mathfrak{f}_1\mathfrak{d}_1^2$, with $\mathfrak{f}_1$ the discriminant of the quadratic field $\mathbb{Q}(\sqrt{\Delta(\uL)})$ 
 and $\mathfrak{d}_1$ in $\mathbb{N}$. For each fixed pair $(D,x)$ in supp$(\uL)$, define $\tilde{D}:=N_{x}^2D$. Let $\tilde{L}_p$ be the local Euler factors defined in \eqref{eq:tildeL} and let $\chi_{\uL}(\cdot)$ denote the quadratic character $\chi_{\uL}(1,\cdot)$ defined in \eqref{eq:chiL}. For even $\rk$, we have:
\begin{equation*}
\begin{split}
G_{0}(D,x)=&
\frac{2(-1)^{\lceil\frac{\rk}{4}\rceil}(-D|\mathfrak{f}_1|)^{k-\frac{\rk}{2}-1}}{\mathfrak{d}_1 L\left(1-k+\frac{\rk}{2},\chi_{\mathfrak{f}_1}\right)\sum_{d\mid \mathfrak{d}_1}\mu(d)\chi_{\mathfrak{f}_1}(d)d^{\frac{\rk}{2}-k}\sigma_{1-2k+\rk}\left(\frac{\mathfrak{d}_1}{d}\right)}\\
&\times\prod_{p\mid 2\tilde{D}\det(\uL)}\frac{\tilde{L}_p(k-1)}{1-\chi_{\uL}(p)p^{\frac{\rk}{2}-k}}.
\end{split}
\end{equation*}
When $\rk$ is odd, write $D=D_0f^2$, with $D_0$ in $\mathbb{Q}_{\leq0}$ and $f$ in $\mathbb{N}$ such that $\gcd(f,2\det(\uL))=1$ and $\ord_{p}(D_0)\in\{0,1\}$ for all primes $p$ that are coprime to $2\det(\uL)$. Define $\tilde{D}_{0}:=D_0N_{x}^2$. Write $\tilde{D}_0\Delta(\uL)=\mathfrak{f}_2\mathfrak{d}_2^2$, with $\mathfrak{f}_2$ the discriminant of the quadratic field $\mathbb{Q}(\sqrt{\tilde{D}_0\Delta(\uL)})$ and $\mathfrak{d}_2$ in $\mathbb{N}$. 
For odd $\rk$, we have:
\begin{equation*}
\begin{split}
G_{0}(D,x)&=\frac{2^{2k-\rk}\left(k-\lceil\frac{\rk}{2}\rceil\right)(D\tilde{D}_0)^{\frac{1}{2}}(-D)^{k-\lceil\frac{\rk}{2}\rceil-1}}{(-1)^{\lceil\frac{\rk}{2}\rceil+\lfloor\frac{\rk}{4}\rfloor}B_{2k-\rk-1}\mathfrak{d}_2|\mathfrak{f}_2|^{k-\lceil\frac{\rk}{2}\rceil}}L\left(1-k+\lceil\frac{\rk}{2}\rceil,\chi_{\mathfrak{f}_2}\right)\\
&\times \sum_{d\mid \mathfrak{d}_2}\mu(d)\chi_{\mathfrak{f}_2}(d)d^{\lceil\frac{\rk}{2}\rceil-k}\sigma_{2-2k+\rk}\left(\frac{\mathfrak{d}_2}{d}\right)\prod_{p\mid 2\tilde{D}\det(\uL)}\frac{1-\chi_{\uL}(\tilde{D}_0,p)p^{\lceil\frac{\rk}{2}\rceil-k}}{1-p^{1-2k+\rk}}\tilde{L}_p(k-1).
\end{split}
\end{equation*}
\end{theorem} 

We will use the following lemma in the proof of Theorem \ref{T:E0}:

\begin{lemma}\label{L:G0simplified}
The non-singular Fourier coefficients of $E_0$ are equal to
\begin{equation}\label{eq:G0simplified}
\begin{split}
G_{0}(D,x)&=\frac{(2\pi)^{k-\frac{\rk}{2}}i^{k}(-D)^{k-\frac{\rk}{2}-1}}{2\det(\uL)^{-\frac{1}{2}}\Gamma\left(k-\frac{\rk}{2}\right)\zeta(k-\rk)}\sum_{b\geq1}\left(1+(-1)^k\right)\frac{R_b(Q_{D,x})}{b^{k-1}}.
\end{split}
\end{equation}
\end{lemma}

\begin{proof}
On the right-hand side of \eqref{eq:Heis}, set $\lambda':=d^{-1}\lambda$. Change the notation of the pair $(D',r')$ to $(D,x)$ and drop the prime from the new summation over $\lambda'$. We obtain:
\begin{equation*}
\begin{split}
H_{\uL,c}(r,D,x)&=
\sum_{\lambda(c)}\sum_{d(c)^{\times}}e_c\left(d\left(\beta(\lambda+d^{-1}r+x)-D\right)\right).
\end{split}
\end{equation*}
We have used the fact that $r\in\iso$ to write $e_c(\beta(\lambda+r))=e_c(\beta(\lambda+d^{-1}dr))$ and $e_c(\beta(\lambda,r))=e_c(\beta(\lambda,d^{-1}dr))$ (note that this would not work for the lattice sum $H_{\uL,c}(D,r,D',r')$). Suppose that $r=0$ in $\iso$ and set $Q_{D,x}(\lambda):=\beta(\lambda+x)-D$. Note that this quantity is an integer. Define the following representation numbers:
\begin{equation*}
R_c(Q_{D,x}):=\#\{\lambda\in L/cL:Q_{D,x}(\lambda)\equiv0\bmod c\}.
\end{equation*} 
We want to use the following identity, which is easy to verify:
\begin{equation*}
R_c(Q_{D,x})=\frac{1}{c}\sum_{\lambda\in L/cL}\sum_{d\in\mathbb{Z}_c}e_c(dQ_{D,x}(\lambda)).
\end{equation*}
To remove the coprimality condition between $d$ and $c$, use the following well-known identity involving the M\"obius function:
\begin{equation*}
\sum_{a\mid n}\mu(a)=\begin{cases}
1, & \text{if }n=1\text{ and }\\
0, & \text{otherwise.} 
\end{cases}
\end{equation*}
Define the following quantity:
\begin{equation*}
\overline{d}=\begin{cases}
d^{-1}\bmod c, & \text{if }\gcd(d,c)=1\text{ and }\\
0, & \text{otherwise.}
\end{cases}
\end{equation*}
We obtain:
\begin{equation*}
H_{\uL,c}(r,D,x)=\sum_{\lambda(c)}\sum_{d(c)}\sum_{a\mid(d,c)}\mu(a)e_c\left(d\left(\beta(\lambda+\overline{d}r+x)-D\right)\right)
\end{equation*}
and, writing $c=ab$, we have:
\begin{equation*}
\begin{split}
\sum_{c\geq1}c^{-k}H_{\uL,c}(r,D,x)&=\sum_{a\geq1}\sum_{b\geq1}\frac{\mu(a)}{(ab)^k}\sum_{\substack{d(ab)\\a\mid d}}\sum_{\lambda(ab)}e_{ab}\left(d\left(\beta(\lambda+\overline{d}r+x)-D\right)\right).
\end{split}
\end{equation*}
The condition that $d$ runs modulo $ab$ and $d$ is divisible by $a$ is equivalent to $\frac{d}{a}$ running modulo $b$.
Since $e_{b}\left(\frac{d}{a}\left(\beta(\lambda+\overline{d}r+x)-D\right)\right)$ only depends on $\lambda$ modulo $b$, we can rewrite the inner sum above as
\begin{equation*}
\sum_{\lambda\in (ab)}e_{b}\left(\frac{d}{a}\left(\beta(\lambda+\overline{d}r+x)-D\right)\right)=a^{\rk}\sum_{\lambda (b)}e_{b}\left(\frac{d}{a}\left(\beta(\lambda+\overline{d}r+x)-D\right)\right).
\end{equation*}
Combining everything with the following well-known identity involving the Riemann zeta function:
\begin{equation*}
\frac{1}{\zeta(s)}=\sum_{n\geq1}\frac{\mu(n)}{n^s},
\end{equation*}
we obtain:
\begin{equation}\label{eq:LfnKl}
\begin{split}
\sum_{c\geq1}c^{-k}H_{\uL,c}(0,D,x)&=\sum_{a\geq1}\frac{\mu(a)}{a^{k-\rk}}\sum_{b\geq1}\frac{1}{b^{k-1}}\cdot \frac{1}{b}\sum_{d(b)}\sum_{\lambda(b)}e_{b}\left(dQ_{D,x}(\lambda)\right)\\
&=\frac{1}{\zeta(k-\rk)}\sum_{b\geq1}\frac{R_b(Q_{D,x})}{b^{k-1}},
\end{split}
\end{equation}
where in the first line we replaced $\frac{d}{a}$ by $d$. 
Therefore, equation $\eqref{eq:GEisenstein}$ with $r=0$ yields the desired result.
\end{proof}

We proceed with the proof of Theorem \ref{T:E0}:

\begin{proof}[Proof of Theorem \ref{T:E0}]

For every $r$ in $\iso$, the singular term of $E_{r}$ can be written as
\begin{equation*}
 C_0(E_{r})(\tau,z)=\frac{1}{2}\left(\vartheta_{\uL,r}+(-1)^k\vartheta_{\uL,-r}\right)(\tau,z).
\end{equation*}
When $r=0$, this quantity is equal to zero if $k$ is odd and to $\vartheta_{\uL,0}$ if $k$ is even. Lemma \ref{L:G0simplified} implies that the Fourier coefficients $G_{0}(D,x)$ vanish if $k$ is odd, in support of \eqref{eq:Er-r}. Thus, from now on, assume that $k$ is even. In the remainder of this proof, we want to bring \eqref{eq:G0simplified} to the desired forms. 

The representation numbers $R_b:=R_b(Q_{D,x})$ are multiplicative functions of $b$. They also arise in $\S4$ of \cite{BrKu} in the context of vector-valued Eisenstein series ($R_b(Q_{D,x})=N_{x,-D}(b)$ with the notation in \cite{BrKu}). Define the Dirichlet series
\begin{equation*}
\tilde{L}(s):=\sum_{b\geq1}\frac{R_b}{b^s}.
\end{equation*}
It was shown in \cite{BrKu} that $\tilde{L}(s)$ converges for $\Re(s)>\rk$ and that it can be continued meromorphically to $\Re(s)>\frac{\rk}{2}+1$, with a simple pole at $s=\rk$, in view of \eqref{eq:LfnKl}. Thus, $G_{0}(D,x)$ is the value of the analytic continuation of 
\begin{equation*}
\frac{(2\pi)^{k-\frac{\rk}{2}}i^{k}(-D)^{k-\frac{\rk}{2}-1}\tilde{L}(s)}{\det(\uL)^{\frac{1}{2}}\Gamma\left(k-\frac{\rk}{2}\right)\zeta(s-\rk+1)}
\end{equation*}
at $s=k-1$. By using results of Siegel on representation numbers of quadratic forms modulo prime powers \cite{Si}, it is possible to compute $\tilde{L}$. For each prime $p$, define $w_p:=1+2\ord_p(2N_{x}D)$ and the local Euler factor
\begin{equation}\label{eq:tildeL}
\tilde{L}_p(s):=p^{-w_ps}R_{p^{w_p}}+\left(1-p^{-(s-\rk+1)}\right)\sum_{l=0}^{w_p-1}p^{-ls}R_{p^l}.
\end{equation} 
Then Lemma $5$ in \cite{BrKu} implies that
\begin{equation*}
\tilde{L}(s)=\zeta(s-\rk+1)\prod_{p\text{ prime}}\tilde{L}_p(s).
\end{equation*}
Define $\tilde{D}:=DN_x^2$, which is a negative integer since $D\equiv\beta(x)\bmod \mathbb{Z}$ and $N_xx\in L$. If the rank of $\uL$ is even, then $\Delta(\uL)$ is a discriminant, i.e. it is congruent to $0$ or $1$ modulo $4$, and hence $\chi_{\uL}(\cdot)=\left(\frac{\Delta(\uL)}{\cdot}\right)$ is a quadratic character of modulus $|\Delta(\uL)|$. Define $\tilde{D}_0$ in the following way:  let $D=D_0f^2$, with $D_0$ in $\mathbb{Q}_{\leq0}$ and $f$ in $\mathbb{N}$ such that $\gcd(f,2\det(\uL))=1$ and $\ord_{p}(-D_0)\in\{0,1\}$ for all primes $p$ that are coprime to $2\det(\uL)$; let $\tilde{D}_{0}:=D_0N_{x}^2$ (which is also a negative integer). If the rank of $\uL$ is odd, then $\Delta(\uL)$ is congruent to $0$ modulo $4$ (see Remark $14.3.23$ in $\S14.3$ of \cite{CS}) and hence $\chi_{\uL}(\tilde{D}_0,\cdot)=\left(\frac{\tilde{D}_0\Delta(\uL)}{\cdot}\right)$ is a quadratic character of modulus $|\tilde{D}_0\Delta(\uL)|$. The local Euler factors $\tilde{L}_p$ can be computed at all primes except for a set of ``bad primes'', by using Hilfssatz $16$ in \cite{Si}, giving rise to the following formulas:
\begin{equation*}
\tilde{L}(s)=\frac{\zeta(s-\rk+1)}{L\left(s-\frac{\rk}{2}+1,\chi_{\uL}(1,\cdot)\right)}\prod_{p\mid 2\tilde{D}\det(\uL)}\frac{\tilde{L}_p(s)}{1-\chi_{\uL}(1,p)p^{-(s-\frac{\rk}{2}+1)}},
\end{equation*}
if $\rk$ is even and
\begin{equation*}
\begin{split}
\tilde{L}(s)&=\frac{\zeta(s-\rk+1)L\left(s-\lfloor\frac{\rk}{2}\rfloor,\chi_{\uL}(\tilde{D}_0,\cdot)\right)}{\zeta(2s-\rk+1)}\prod_{p\mid 2\tilde{D}\det(\uL)}\frac{1-\chi_{\uL}(\tilde{D}_0,p)p^{-\left(s-\lfloor\frac{\rk}{2}\rfloor\right)}}{1-p^{-(2s-\rk+1)}}\tilde{L}_p(s),
\end{split}
\end{equation*}
if $\rk$ is odd.

Assume first that $\rk$ is even. Write $\Delta(\uL)=\mathfrak{f}\mathfrak{d}^2$ with $\mathfrak{f}$ the discriminant of the quadratic field $\mathbb{Q}(\sqrt{\Delta(\uL)})$ and $\mathfrak{d}$ in $\mathbb{N}$. In particular, the function $\chi_{\mathfrak{f}}(\cdot)$ is a primitive quadratic character modulo $|\mathfrak{f}|$. It was shown in $\S4$ of \cite{Za} that the Dirichlet $L$-function of $\chi_{\uL}$ satisfies the following:
\begin{equation*}
L(s,\chi_{\uL})=L(s,\chi_{\mathfrak{f}})\sum_{d\mid \mathfrak{d}}\mu(d)\chi_{\mathfrak{f}}(d)d^{-s}\sigma_{1-2s}\left(\frac{\mathfrak{d}}{d}\right).
\end{equation*}
Write \eqref{eq:G0simplified} as $G_{0}(D,x)=A\cdot B$, with
\begin{equation*}
A:=\frac{2^{k-\frac{\rk}{2}}i^{k}(-D)^{k-\frac{\rk}{2}-1}}{\sum_{d\mid \mathfrak{d}}\mu(d)\chi_{\mathfrak{f}}(d)d^{\frac{\rk}{2}-k}\sigma_{1-2k+\rk}\left(\frac{\mathfrak{d}}{d}\right)}\prod_{p\mid 2\tilde{D}\det(\uL)}\frac{\tilde{L}_p(k-1)}{1-\chi_{\uL}(p)p^{-(k-\frac{\rk}{2})}}
\end{equation*} 
and 
\begin{equation*}
B:=\frac{\pi^{k-\frac{\rk}{2}}}{\det(\uL)^{\frac{1}{2}}\Gamma\left(k-\frac{\rk}{2}\right)L\left(k-\frac{\rk}{2},\chi_{\mathfrak{f}}\right)}.
\end{equation*}
Since $k$ and $\rk$ are even, it follows that $A$ is a rational number. We can rewrite the expression for $B$ using functional equations for Dirichlet $L$-series (see $\S3.4.3$ of \cite{CS}): assume that $\chi$ is a primitive character modulo $N$ and define the Gauss sum $G(\chi):=\sum_{n=1}^N\chi(n)e_N(n)$ and the quantity
\begin{equation*}
a:=\begin{cases}
0, & \text{if }\chi(-1)=1\text{ and }\\
1, & \text{if }\chi(-1)=-1.
\end{cases}
\end{equation*}
Define the completed $L$-function of $\chi$,
\begin{equation*}
\Lambda(s,\chi):=\left(\frac{N}{\pi}\right)^{\frac{s+a}{2}}\Gamma\left(\frac{s+a}{2}\right)L(s,\chi).
\end{equation*}
The following holds:
\begin{equation}\label{eq:dirfneq}
\Lambda(1-s,\chi)=\frac{G(\chi)}{((-1)^aN)^{\frac{1}{2}}}\Lambda(s,\overline{\chi}),
\end{equation}
Applying this functional equation, we obtain:
\begin{equation*}
\begin{split}
B=\frac{i^{a}\pi^{\frac{1}{2}}|\mathfrak{f}|^{k-\frac{\rk}{2}}}{\det(\uL)^{\frac{1}{2}}G(\chi_{\mathfrak{f}})L\left(1-k+\frac{\rk}{2},\chi_{\mathfrak{f}}\right)}\times\frac{\Gamma\left(\frac{k}{2}-\frac{\rk}{4}+\frac{a}{2}\right)}{\Gamma\left(k-\frac{\rk}{2}\right)\Gamma\left(\frac{1+a}{2}-\frac{k}{2}+\frac{\rk}{4}\right)}.
\end{split}
\end{equation*}
We remind the reader that $\Delta(\uL)=(-1)^{\frac{\rk}{2}}\det(\uL)$ and, since $\Delta(\uL)=\mathfrak{f}\mathfrak{d}^2$, we have $\mathfrak{f}>0$ if $\rk\equiv0\bmod4$ and $\mathfrak{f}<0$ if $\rk\equiv2\bmod4$. The Kronecker symbol $\left(\frac{n}{-1}\right)$ is equal to $\text{sign}(n)$ and so $a=0$ if $\rk\equiv0\bmod4$ and $a=1$ if $\rk\equiv2\bmod4$. The Gauss sum $G(\chi_{\mathfrak{f}})$ is equal to $\mathfrak{f}^{\frac{1}{2}}$, since $\mathfrak{f}$ is a fundamental discriminant (see $\S3.4.2$ of \cite{CS}), and $\det(\uL)^{\frac{1}{2}}\mathfrak{f}^{\frac{1}{2}}=((-1)^{\frac{\rk}{2}}\mathfrak{f}^2\mathfrak{d}^2)^{\frac{1}{2}}=i^a|\mathfrak{f}|\mathfrak{d}$. Using the duplication formula and Euler's reflection formula for the Gamma function and the fact that $\sin\left(\pi\left(n+\frac{1}{2}\right)\right)=(-1)^{n}$, the Gamma functions give a contribution of
\begin{equation*}
2^{1-\left(k-\frac{\rk}{2}\right)}\pi^{-\frac{1}{2}}(-1)^{\frac{k}{2}-\lceil\frac{\rk}{4}\rceil}.
\end{equation*}
It follows that
\begin{equation*}
B=\frac{2^{1-\left(k-\frac{\rk}{2}\right)}(-1)^{\frac{k}{2}-\lceil\frac{\rk}{4}\rceil}|\mathfrak{f}|^{k-\frac{\rk}{2}-1}}{\mathfrak{d} L\left(1-k+\frac{\rk}{2},\chi_{\mathfrak{f}}\right)}
\end{equation*}
and hence we obtain the desired formula for $G_{0}(D,x)$ when $\rk$ is even.

Assume now that $\rk$ is odd. Write $\tilde{D}_0\Delta(\uL)=\mathfrak{f}\mathfrak{d}^2$ with $\mathfrak{f}$ the discriminant of the quadratic field $\mathbb{Q}(\sqrt{\tilde{D}_0\Delta(\uL)})$ and $\mathfrak{d}$ in $\mathbb{N}$. The function $\chi_{\mathfrak{f}}(\cdot)=\left(\frac{\mathfrak{f}}{\cdot}\right)$ is a primitive quadratic character modulo $|\mathfrak{f}|$. As before, the $L$-function of $\chi_{\uL}(\tilde{D}_0,\cdot)$ satisfies the following:
\begin{equation*}
L(s,\chi_{\uL}(\tilde{D}_0,\cdot))=L(s,\chi_{\mathfrak{f}})\sum_{d\mid \mathfrak{d}}\mu(d)\chi_{\mathfrak{f}}(d)d^{-s}\sigma_{1-2s}\left(\frac{\mathfrak{d}}{d}\right).
\end{equation*}
The values of the Riemann zeta function at positive even integers are well known:
\begin{equation*}
\zeta(2k-\rk-1)=\frac{(-1)^{k-\frac{\rk-1}{2}}B_{2k-\rk-1}(2\pi)^{2k-\rk-1}}{2\Gamma(2k-\rk)}.
\end{equation*}
Write \eqref{eq:G0simplified} as $G_{0}(D,x)=A\cdot B$, with
\begin{equation*}
\begin{split}
A&:=\frac{2i^k(-D)^{k-\lceil\frac{\rk}{2}\rceil-1}\sum_{d\mid \mathfrak{d}}\mu(d)\chi_{\mathfrak{f}}(d)d^{\lceil\frac{\rk}{2}\rceil-k}\sigma_{2-2k+\rk}\left(\frac{\mathfrak{d}}{d}\right)}{(-1)^{k-\frac{\rk-1}{2}}B_{2k-\rk-1}}\\
&\times\prod_{p\mid 2\tilde{D}\det(\uL)}\frac{1-\chi_{\uL}(\tilde{D}_0,p)p^{-\left(k-\lceil\frac{\rk}{2}\rceil\right)}}{1-p^{-(2k-\rk-1)}}\tilde{L}_p(k-1)
\end{split}
\end{equation*}
and note that $A$ is a rational number. Applying the functional equation \eqref{eq:dirfneq} to $\chi_\mathfrak{f}$ yields
\begin{equation*}
\begin{split}
B&=\frac{(2\pi)^{k-\frac{\rk}{2}}(-D)^{\frac{1}{2}}\Gamma\left(2k-\rk\right)L\left(k-\lceil\frac{\rk}{2}\rceil,\chi_{\mathfrak{f}}\right)}{\det(\uL)^{\frac{1}{2}}\Gamma\left(k-\frac{\rk}{2}\right)(2\pi)^{2k-\rk-1}}\\
&=\frac{(-D)^{\frac{1}{2}}G(\chi_{\mathfrak{f}})L\left(1-k+\lceil\frac{\rk}{2}\rceil,\chi_{\mathfrak{f}}\right)}{i^a\det(\uL)^{\frac{1}{2}}2^{k-\frac{\rk}{2}-1}|\mathfrak{f}|^{k-\lceil\frac{\rk}{2}\rceil}}\times\frac{\Gamma\left(\frac{1+a}{2}-\frac{k}{2}+\frac{\rk+1}{4}\right)\Gamma\left(2k-\rk\right)}{\Gamma\left(k-\frac{\rk}{2}\right)\Gamma\left(\frac{k}{2}-\frac{\rk+1}{4}+\frac{a}{2}\right)}.
\end{split}
\end{equation*}
We remind the reader that $\Delta(\uL)=(-1)^{\lfloor\frac{\rk}{2}\rfloor}2\det(\uL)$ and, since $\tilde{D}_0\Delta(\uL)=\mathfrak{f}\mathfrak{d}^2$ and $\tilde{D}_0<0$, we have $\mathfrak{f}>0$ if $\rk\equiv3\bmod4$ and $\mathfrak{f}<0$ if $\rk\equiv1\bmod4$. It follows that $a=0$ if $\rk\equiv3\bmod4$ and $a=1$ if $\rk\equiv1\bmod4$. The Gauss sum $G(\chi_{\mathfrak{f}})$ is equal to $\mathfrak{f}^{\frac{1}{2}}$ as before and $\mathfrak{f}^{\frac{1}{2}}/\det(\uL)^{\frac{1}{2}}=((-1)^{\lfloor\frac{\rk}{2}\rfloor}2\tilde{D}_0/\mathfrak{d}^2)^{\frac{1}{2}}=i^a(-\tilde{D}_0)^{\frac{1}{2}}2^{\frac{1}{2}}/\mathfrak{d}$. Using the duplication formula and Euler's reflection formula, the Gamma functions give a contribution of
\begin{equation*}
\begin{split}
\frac{(-1)^{\frac{k}{2}-\lfloor\frac{\rk}{4}\rfloor-1}\left(k-\frac{\rk}{2}-\frac{1}{2}\right)}{2^{\frac{5}{2}-3k+\frac{3\rk}{2}}}.
\end{split}
\end{equation*}
It follows that
\begin{equation*}
B=\frac{2^{2k-\rk-1}\left(k-\lceil\frac{\rk}{2}\rceil\right)(D\tilde{D}_0)^{\frac{1}{2}}L\left(1-k+\lceil\frac{\rk}{2}\rceil,\chi_{\mathfrak{f}}\right)}{(-1)^{\frac{k}{2}+\lfloor\frac{\rk}{4}\rfloor+1}\mathfrak{d}|\mathfrak{f}|^{k-\lceil\frac{\rk}{2}\rceil}}
\end{equation*}
and hence we obtain the desired formula for $G_{0}(D,x)$ when $\rk$ is odd, completing the proof.

\end{proof}

\begin{remark}
The local Euler factors $\tilde{L}_p$ can be computed at the set of ``bad primes'' $p\mid2\tilde{D}\det(\uL)$ by using the methods for calculating the Igusa local zeta function from \cite{CKW}. For example, let $\rk$ be even and set 
\begin{equation*}
\mathfrak{D}:=\prod_{\substack{p\mid\tilde{D}\\\gcd(p,2\det(\uL))=1}}p^{\ord_p(\tilde{D})}.
\end{equation*}
Then Theorem $2.1$ in \cite{CKW} can be used to show that
\begin{equation*}
\begin{split}
\prod_{\substack{p\mid \tilde{D}\\p\nmid 2\det(\uL)}}\frac{\tilde{L}_p(k-1)}{1-\chi_{\mathfrak{f}_1}(p)p^{-\left(k-\frac{\rk}{2}\right)}}&=\chi_{\mathfrak{f}_1}(\mathfrak{D})\mathfrak{D}^{-\left(k-\frac{\rk}{2}-1\right)}\sigma_{k-\frac{\rk}{2}-1}^{\chi_{\mathfrak{f}_1}}(\mathfrak{D}).
\end{split}
\end{equation*} 
The proof involves tedious calculations, hence we omit it.

\end{remark}

An important consequence of Theorem \ref{T:E0} is the following rationality result:

\begin{corollary}
The Fourier coefficients of $E_{0}$ are rational numbers.
\end{corollary}

\begin{proof}

The values of $L(s,\overline{\chi_\mathfrak{f}})$ at negative integers can be expressed as linear combinations of Bernoulli polynomials 
 with rational coefficients, cf. $\S1.7$ of \cite{Za2}:
\begin{equation*}
L\left(-n,\chi_{\mathfrak{f}}\right)=-\frac{|\mathfrak{f}|^{n}}{n+1}\sum_{j=1}^{|\mathfrak{f}|}\chi_{\mathfrak{f}}(j)B_{n+1}\left(\frac{j}{|\mathfrak{f}|}\right).
\end{equation*}
Furthermore, when $\rk$ is odd, the definition of $\tilde{D}_0$ gives $(D\tilde{D}_0)^{\frac{1}{2}}=(DD_0N_x^2)^{\frac{1}{2}}=-\frac{DN_x}{f}$. Since all other quantities appearing in the formulas are clearly rational, the result follows.

\end{proof}

\section{Non-trivial Eisenstein series}\label{S:Er}

Let $r$ be an element of $\iso\setminus\{0\}$. Our goal is to express $E_{r}$ in terms of $E_{0}$ using an approach based on vector-valued modular forms, which were introduced in Section \ref{Ss:VV}. Let $x$ be an element of $L^\#/L$ with order $N_x$ and let $\sigma_x$ denote the Schr\"odinger representation twisted at $x$, defined in \eqref{eq:schrodef}. Let $\phi$ be an element of $J_{k,\uL}$ and let $\varphi$ be the isomorphism between Jacobi forms and vector-valued modular forms given in Theorem \ref{T:jacobivv}.  Define the {\em averaging operator at $x$} in the following way:
\begin{equation}\label{eq:averagingop}
\Av_x\phi(\tau,z):=\frac{1}{N_x^2}\sum_{(\lambda,\mu)\in(\mathbb{Z}_{N_x^2})^2}\left(\varphi^{-1}\sigma_x^*(\lambda,\mu,0)\varphi\right)\phi(\tau,z).
\end{equation}
This operator was defined for vector-valued modular forms by Williams in $\S11$ of \cite{Wi}.

\begin{remark}
It was shown in $\S3.5$ of \cite{Bo} that the theta series $\vartheta_{\uL,y}$ ($y\in L^\#/L$) are linearly independent. The map $\vartheta_{\uL,y}\mapsto y$ defines a canonical isomorphism between $\text{span}_{\mathbb{C}}\{\vartheta_{\uL,y}:y\in L^\#/L\}$ and $\CL$. Thus, the action of the Schr\"odinger representation can be defined directly on theta series:
\begin{equation*}
\sigma_x(\lambda,\mu,t)\vartheta_{\uL,y}=e(\mu\beta(x,y)+(t-\lambda\mu)\beta(x))\vartheta_{\uL,y-\lambda x}
\end{equation*}
and hence the operator $\Av_x$ can be defined without the use of $\varphi$. However, we continue to work with vector-valued modular forms, since it is easier to prove modularity in this context.
\end{remark}
 
The following holds: 

 \begin{lemma}
 The operator $\Av_x$ is well defined, in other words it does not depend on the choice of representatives of $\mathbb{Z}_{N_x^2}$. Furthermore, it maps $J_{k,\uL}$ to $J_{k,\uL}$.
 \end{lemma}
 
 \begin{proof}
For all integers $a$ and $b$, we have
\begin{equation*}
\sigma_x^*(\lambda+aN_x^2,\mu,0)\e_{y}=\sigma_x^*(\lambda,\mu+bN_x^2,0)\e_{y}=\sigma_x^*(\lambda,\mu,0)\e_{y}.
\end{equation*}
Thus, the operator $\Av_x$ is well defined. To show that $\Av_x\phi$ is a Jacobi form of weight $k$ and index $\uL$, it is sufficient to prove that $\varphi \Av_x\phi$ is an element of $M_{k-\frac{\rk}{2}}(\rho_{\uL}^*)$, in view of Theorem \ref{T:jacobivv}. Fix a pair $(\lambda,\mu)$ in $(\mathbb{Z}_{N_x^2})^2$ and for simplicity let $F(\tau):=\varphi(\phi)\in M_{k-\frac{\rk}{2}}(\rho_{\uL}^*)$ and $\psi(\tau):=\sigma_x^*(\lambda,\mu,0)F(\tau)$. Applying the $|_{k-\frac{\rk}{2}}$ action of an element $\tilde{A}=\left(A,w(\tau)\right)$ in $\Mp_2(\mathbb{Z})$ on $\psi(\tau)$ gives the following:
\begin{equation*}
\begin{split}
\psi|_{k-\frac{\rk}{2}}\tilde{A}(\tau)&=w(\tau)^{-2\left(k-\frac{\rk}{2}\right)}\psi(A\tau)=w(\tau)^{-2\left(k-\frac{\rk}{2}\right)}\sigma_x^*(\lambda,\mu,0)F(A\tau)\\
&=w(\tau)^{-2\left(k-\frac{\rk}{2}\right)}\rho_{\uL}^*(\tilde{A})\sigma_x^*((\lambda,\mu,0)^A)\rho_{\uL}^*(\tilde{A})^{-1}F(A\tau)\\
&=\rho_{\uL}^*(\tilde{A})\sigma_x^*((\lambda,\mu,0)^A)F(\tau),
\end{split}
\end{equation*}
where we have used \eqref{eq:weilschroconj} in the middle line. Thus,
\begin{align*}
(\varphi\Av_x\phi)|_{k-\frac{\rk}{2}}\tilde{A}(\tau)&=\frac{1}{N_x^2}\sum_{(\lambda,\mu)\in(\mathbb{Z}_{N_x^2})^2}\rho_{\uL}^*(\tilde{A})\sigma_x^*((\lambda,\mu,0)^A)F(\tau)\\
&=\frac{1}{N_x^2}\sum_{(\lambda',\mu')\in(\mathbb{Z}_{N_x^2})^2}\rho_{\uL}^*(\tilde{A})\sigma_x^*(\lambda',\mu',0)F(\tau)\\
\nonumber&=\rho_{\uL}^*(\tilde{A})(\varphi\Av_x\phi)(\tau),
\end{align*}
with the bijective change of variable $(\lambda',\mu')=((\lambda,\mu)A)$. Thus, the function $\varphi\Av_x\phi$ is an element of $M_{k-\frac{\rk}{2}}(\rho_{\uL}^*)$ and we can apply $\varphi^{-1}$ to obtain that $\Av_x\phi$ is an element of $J_{k,\uL}$.
\end{proof}

Applying the averaging operator at $x$ to the Eisenstein series $E_{0}$ yields the following:

\begin{proposition}\label{P:ave0}
Let $x$ be an element of $L^\#/L$ and let the averaging operator $\Av_x$ be defined as above. Then:
\begin{equation}\label{eq:ave0}
\Av_x E_{0}(\tau,z)=\sum_{\substack{\lambda\in\mathbb{Z}_{N_x^2}\\ \lambda\beta(x)\in\mathbb{Z}}}E_{\lambda x}(\tau,z).
\end{equation}
\end{proposition}

\begin{proof}
First, note that, if $\lambda\beta(x)\in\mathbb{Z}$, then $\beta(\lambda x)=\lambda^2\beta(x)\in\mathbb{Z}$ and so $\lambda x\in \iso$. Thus, the right-hand side of \eqref{eq:ave0} is well defined. Expand $\eqref{eq:Eistheta}$ and use \eqref{eq:thetamod}:
\begin{align}\nonumber
E_{r}(\tau,z)&=\frac{1}{2}\sum_{A\in\Gamma_{\infty}\setminus\Gamma}(c \tau+d)^{-k}e\left(\frac{-cz}{c\tau+d}\right)\vartheta_{\uL,r}(A(\tau,z))\\
\nonumber&=\frac{1}{2}\sum_{A\in\Gamma_{\infty}\setminus\Gamma}1|_{k-\frac{\rk}{2}}\tilde{A}(\tau)\vartheta_{\uL,r}|_{\frac{\rk}{2},\uL}\tilde{A}(\tau,z)\\
\label{eq:eisthetah}&=\frac{1}{2}\sum_{A\in\Gamma_{\infty}\setminus\Gamma}1|_{k-\frac{\rk}{2}}\tilde{A}(\tau)\sum_{y\in L^\#/L}\rho_{\uL}(\tilde{A})_{r,y}\vartheta_{\uL,y}(\tau,z)\\
\nonumber&=\frac{1}{2}\sum_{A\in\Gamma_{\infty}\setminus\Gamma}1|_{k-\frac{\rk}{2}}\tilde{A}(\tau)\rho_{\uL}(\tilde{A})^t\vartheta_{\uL,r}(\tau,z).
\end{align}
 The fact that $\rho_{\uL}$ is unitary implies that $\rho_{\uL}^*(\tilde{A})^{-1}=\rho_{\uL}(\tilde{A})^t$ and hence
\begin{equation*}
\begin{split}
\Av_xE_{0}(\tau,z)&=\frac{1}{N_x^2}\sum_{(\lambda,\mu)\in(\mathbb{Z}_{N_x^2})^2}\varphi^{-1}\sigma_x^*(\lambda,\mu,0)\frac{1}{2}\sum_{A\in\Gamma_{\infty}\setminus\Gamma}1|_{k-\frac{\rk}{2}}\tilde{A}(\tau)\rho_{\uL}^*(\tilde{A})^{-1}\e_{0}.
\end{split}
\end{equation*}
Using \eqref{eq:weilschroconj}, we have:
\begin{equation*}
\begin{split}
\sum_{(\lambda,\mu)\in(\mathbb{Z}_{N_x^2})^2}\sigma_x^*(\lambda,\mu,0)\rho_{\uL}^*(\tilde{A})^{-1}\e_{0}&=\sum_{(\lambda,\mu)\in(\mathbb{Z}_{N_x^2})^2}\rho_{\uL}^*(\tilde{A})^{-1}\sigma_x^*\left((\lambda,\mu,0)^{A^{-1}}\right)\e_{0}\\
&=\sum_{(\lambda',\mu')\in(\mathbb{Z}_{N_x^2})^2}\rho_{\uL}^*(\tilde{A})^{-1}\sigma_x^*\left(\lambda',\mu',0\right)\e_{0}\\
&=\sum_{(\lambda,\mu)\in(\mathbb{Z}_{N_x^2})^2}e(\lambda\mu\beta(x))\rho_{\uL}^*(\tilde{A})^{-1}\e_{-\lambda x}\\
&=\sum_{\lambda\in\mathbb{Z}_{N_x^2}}\rho_{\uL}^*(\tilde{A})^{-1}\e_{-\lambda x}\sum_{\mu\in\mathbb{Z}_{N_x^2}}e(\mu\lambda\beta(x)).
\end{split}
\end{equation*}
In the second line, we made the change of variable $(\lambda',\mu')=\left((\lambda,\mu)A^{-1}\right)$. The inner sum in the last line is equal to $N_x^2$ if $\lambda\beta(x)\in\mathbb{Z}$ and to zero if $\lambda\beta(x)$ is not an integer. Thus,
\begin{equation*}
\begin{split}
\Av_xE_{0}(\tau,z)&=\sum_{\substack{\lambda\in\mathbb{Z}_{N_x^2}\\ \lambda\beta(x)\in\mathbb{Z}}}\varphi^{-1}\frac{1}{2}\sum_{A\in\Gamma_{\infty}\setminus\Gamma}1|_{k-\frac{\rk}{2}}\tilde{A}(\tau)\rho_{\uL}(\tilde{A})^{t}\e_{-\lambda x}\\
&=\sum_{\substack{\lambda\in\mathbb{Z}_{N_x^2}\\ \lambda\beta(x)\in\mathbb{Z}}}E_{-\lambda x}(\tau,z)=\sum_{\substack{\lambda\in\mathbb{Z}_{N_x^2}\\ \lambda\beta(x)\in\mathbb{Z}}}E_{\lambda x}(\tau,z),
\end{split}
\end{equation*}
as claimed.
\end{proof}

We remind the reader of \eqref{eq:Er-r}, which asserts that $E_{x}=-E_{-x}$ when $k$ is odd. On the right hand-side of \eqref{eq:ave0}, $\lambda\in\mathbb{Z}_{N_x^2}$ satisfies $\lambda\beta(x)\in\mathbb{Z}$ if and only if $-\lambda\beta(x)\in\mathbb{Z}$. Furthermore, the trivial Eisenstein series $E_{0}\equiv0$ when $k$ is odd and so, in this case, the terms on the right-hand side of \eqref{eq:ave0} cancel out. Thus, from now on, let $k$ be even. 

Note that \eqref{eq:eisthetah} can be rewritten as
\begin{equation*}
\begin{split}
E_{r}(\tau,z)&=\sum_{y\in L^\#/L}h_{r,y}(\tau)\vartheta_{\uL,y}(\tau,z),
\end{split}
\end{equation*}
where $h_{r,y}$ denotea the $y$-th component of the theta expansion of the Eisenstein series $E_{r}$,
\begin{equation}\label{eq:hry}
h_{r,y}(\tau):=h_{E_r,y}(\tau)=\frac{1}{2}\sum_{A\in\Gamma_{\infty}\setminus\Gamma}\rho_{\uL}(\tilde{A})_{r,y}1|_{k-\frac{\rk}{2}}\tilde{A}(\tau).
\end{equation}
Taking $x$ to be an isotropic element in Proposition \ref{P:ave0} gives the main result of this section:

\begin{proposition}\label{P:avxiso}
Suppose that $k$ is even and let $x$ be an element of $\iso$. Then:
\begin{equation}\label{eq:avxiso}
\sum_{\lambda\in\mathbb{Z}_{N_x}}E_{\lambda x}(\tau,z)=\sum_{\substack{y\in L^\#/L\\\beta(x,y)\in\mathbb{Z}}}\left(\sum_{\lambda\in\mathbb{Z}_{N_x}}h_{0,y+\lambda x}(\tau)\right)\vartheta_{\uL,y}(\tau,z).
\end{equation}
\end{proposition}

\begin{proof}
Insert the definition of $\Av_x$ on the left-hand side of \eqref{eq:ave0} and expand:
\begin{equation*}
\begin{split}
\Av_xE_{0}(\tau,z)&=\frac{1}{N_x^2}\varphi^{-1}\sum_{(\lambda,\mu)\in(\mathbb{Z}_{N_x^2})^2}\sum_{y\in L^\#/L}e(\lambda\mu\beta(x)-\mu\beta(x,y))h_{0,y}(\tau)\e_{y-\lambda x}\\
&=\frac{1}{N_x^2}\varphi^{-1}\sum_{y\in L^\#/L}h_{0,y}(\tau)\sum_{\lambda\in\mathbb{Z}_{N_x^2}}\e_{y-\lambda x}\sum_{\mu\in\mathbb{Z}_{N_x^2}}e(-\mu\beta(x,y))\\
&=\varphi^{-1}\sum_{\substack{y\in L^\#/L\\\beta(x,y)\in\mathbb{Z}}}h_{0,y}(\tau)\sum_{\lambda\in\mathbb{Z}_{N_x^2}}\e_{y-\lambda x}\\
&=N_x\varphi^{-1}\sum_{\substack{y\in L^\#/L\\\beta(x,y)\in\mathbb{Z}}}h_{0,y}(\tau)\sum_{\lambda\in\mathbb{Z}_{N_x}}\e_{y-\lambda x},
\end{split}
\end{equation*}
since $\e_{y-\lambda x}$ only depends on $\lambda\bmod N_x$. Setting $y'=y-\lambda x$ and dropping the prime, the last line can be written as
\begin{equation*}
\begin{split}
\Av_xE_{0}(\tau,z)&=N_x\varphi^{-1}\sum_{\substack{y\in L^\#/L\\\beta(x,y)\in\mathbb{Z}}}\left(\sum_{\lambda\in\mathbb{Z}_{N_x}}h_{0,y+\lambda x}(\tau)\right)\e_{y}\\
&=N_x\sum_{\substack{y\in L^\#/L\\\beta(x,y)\in\mathbb{Z}}}\left(\sum_{\lambda\in\mathbb{Z}_{N_x}}h_{0,y+\lambda x}(\tau)\right)\vartheta_{\uL,y}(\tau,z).
\end{split}
\end{equation*}
On the other hand, if $x\in\iso$, then $\lambda\beta(x)\in\mathbb{Z}$ for every $\lambda$ in $\mathbb{Z}_{N_x^2}$ and so the right-hand side of \eqref{eq:ave0} is equal to
\begin{equation*}
\begin{split}
\sum_{\lambda\in\mathbb{Z}_{N_x^2}}E_{\lambda x}(\tau,z)=N_x\sum_{\lambda\in\mathbb{Z}_{N_x}}E_{\lambda x}(\tau,z),
\end{split}
\end{equation*}
since $E_{r}$ only depends on $r\bmod L$. Hence, the proof is complete.
\end{proof}

We list a few examples in which Proposition \ref{P:avxiso} can be used to compute Eisenstein series indexed by elements $x$ in $\iso$ of small order:

\begin{example}
Suppose that $x$ in $\iso$ has order $2$ in $L^\#/L$. Applying Proposition \ref{P:avxiso} to an isotropic element $x$ of order $2$ yields
\begin{equation*}
E_{0}(\tau,z)+E_{x}(\tau,z)=\sum_{\substack{y\in L^{\#}/L\\ \beta(x,y)\in\mathbb{Z}}}\left(h_{0,y}+h_{0,y+x}\right)(\tau)\vartheta_{\uL,y}(\tau,z)
\end{equation*}
and hence the Fourier coefficients of $E_x$ are given by
\begin{equation*}
G_{x}(D,y)=
\begin{cases}
G_{0}(D,y+x), & \text{if }\beta(x,y)\in\mathbb{Z}\text{ and }\\
-G_{0}(D,y), & \text{otherwise.}
\end{cases}
\end{equation*}
\end{example}

\begin{example}
Suppose that $x$ in $\iso$ has order $3$. Using Proposition \ref{P:avxiso} and the fact that $E_{x}=E_{-x}$ when $k$ is even, we obtain:
\begin{equation*}
\begin{split}
E_{0}+E_{x}+E_{2x}&=E_{0}+2E_{x}=\sum_{\substack{y\in L^\#/L\\\beta(x,y)\in\mathbb{Z}}}\left(h_{0,y}+h_{0,y+x}+h_{0,y+2x}\right)\vartheta_{\uL,y}.
\end{split}
\end{equation*}
Therefore,
\begin{equation*}
G_{x}(D,y)=
\begin{cases}
\frac{1}{2}\left(G_{0}(D,y+x)+G_{0}(D,y+2x)\right), & \text{if }\beta(x,y)\in\mathbb{Z}\text{ and }\\
-\frac{1}{2}G_{0}(D,y), & \text{otherwise.}
\end{cases}
\end{equation*}
\end{example}

\begin{example}
If $x$ in $\iso$ has order $4$, then by the same argument as above we have:
\begin{equation*}
\begin{split}
E_{0}+2E_{x}+E_{2x}&=\sum_{\substack{y\in L^\#/L\\\beta(x,y)\in\mathbb{Z}}}\left(h_{0,y}+h_{0,y+x}+h_{0,y+2x}+h_{0,y+3x}\right)\vartheta_{\uL,y}
\end{split}
\end{equation*}
and since $2x$ has order $2$, we have:
\begin{equation*}
G_{2x}(D,y)=
\begin{cases}
G_{0}(D,y+2x), & \text{if }\beta(2x,y)\in\mathbb{Z}\text{ and }\\
-G_{0}(D,y), & \text{otherwise.}
\end{cases}
\end{equation*}
Note that $\beta(x,y)\in\mathbb{Z}$ implies $\beta(2x,y)\in\mathbb{Z}$ and hence
\begin{equation*}
G_{x}(D,y)=
\begin{cases}
\frac{1}{2}\left(G_{0}(D,y+x)+G_{0}(D,y+3x)\right), & \text{if }\beta(x,y)\in\mathbb{Z},\\
-\frac{1}{2}\left(G_{0}(D,y)+G_{0}(D,y+2x)\right), & \text{if }\beta(2x,y)\in\mathbb{Z}\text{ and }\beta(x,y)\not\in\mathbb{Z}\text{ and }\\
0, & \text{otherwise.}
\end{cases}
\end{equation*}
\end{example}

\begin{example}
When $x$ in $\iso$ has order $6$, we can use Proposition \ref{P:avxiso} to obtain a formula for
\begin{equation*}
E_{0}+E_{x}+E_{2x}+E_{3x}+E_{4x}+E_{5x}.
\end{equation*}
Since $3x$ has order $2$, we have:
\begin{equation*}
G_{3x}(D,y)=
\begin{cases}
G_{0}(D,y+3x), & \text{if }\beta(3x,y)\in\mathbb{Z}\text{ and }\\
-G_{0}(D,y), & \text{otherwise.}
\end{cases}
\end{equation*}
We can also use our previous calculations for an element of order $3$ to the Fourier coefficients of  $E_{2x}$ (which is equal to $E_{4x}$):
\begin{equation*}
G_{2x}(D,y)=
\begin{cases}
\frac{1}{2}\left(G_{0}(D,y+2x)+G_{0}(D,y+4x)\right), & \text{if }\beta(2x,y)\in\mathbb{Z}\text{ and }\\
-\frac{1}{2}G_{0}(D,y), & \text{otherwise.}
\end{cases}
\end{equation*}
We obtain a formula for the Fourier coefficients of $E_{x}$ (which is equal to $E_{5x}$). As before, if $\beta(x,y)\in\mathbb{Z}$, then $\beta(2x,y)\in\mathbb{Z}$ and $\beta(3x,y)\in\mathbb{Z}$. If $\beta(x,y)\not\in\mathbb{Z}$ and $\beta(2x,y)\in\mathbb{Z}$, then $\beta(x,y)=a/2$ for some odd $a$ and hence $\beta(3x,y)\not\in\mathbb{Z}$. Similarly, if $\beta(x,y)\not\in\mathbb{Z}$ and $\beta(3x,y)\in\mathbb{Z}$, then $\beta(2x,y)\not\in\mathbb{Z}$. Therefore,
\begin{equation*}
G_{x}(D,y)=
\begin{cases}
\frac{1}{2}\left(G_{0}(D,y+x)+G_{0}(D,y+5x)\right), & \text{if }\beta(x,y)\in\mathbb{Z},\\
-\frac{1}{2}\left(G_{0}(D,y+2x)+G_{0}(D,y+4x)\right), & \text{if }\beta(2x,y)\in\mathbb{Z}\text{ and }\beta(x,y)\not\in\mathbb{Z},\\
-\frac{1}{2}G_{0}(D,y+3x), & \text{if }\beta(3x,y)\in\mathbb{Z}\text{ and }\beta(x,y)\not\in\mathbb{Z}\text{ and }\\
-\frac{1}{2}G_{0}(D,y), & \text{otherwise.}
\end{cases}
\end{equation*}
\end{example}

Note that, when $x$ in $\iso$ has order $5$, we only obtain a formula for the Fourier coefficients of $E_{x}+E_{2x}$:
\begin{equation*}
G_{x}(D,y)+G_{2x}(D,y)=
\begin{dcases}
\substack{\frac{1}{2}(G_{0}(D,y+x)+G_{0}(D,y+2x)\\+G_{0}(D,y+3x)+G_{0}(D,y+4x)), }&  \text{if }\beta(x,y)\in\mathbb{Z}\text{ and }\\
-\frac{1}{2}G_{0}(D,y), & \text{otherwise}.
\end{dcases}
\end{equation*}
By the same reasoning, when $x$ in $\iso$ has odd prime order $p$, we can obtain a formula for $E_{x}+E_{2x}+E_{3x}+\cdots+E_{\frac{p-1}{2}x}$. When $x$ in $\iso$ has order $8$, we can use Proposition \ref{P:avxiso} to obtain a formula for 
\begin{equation*}
E_{0}+E_{x}+E_{2x}+E_{3x}+E_{4x}+E_{5x}+E_{6x}+E_{7x}
\end{equation*}
and we can find $E_{4x}$ and $E_{2x}=E_{6x}$ as before, but we can only obtain a formula for $E_{x}+E_{3x}$ afterwards.

\bibliographystyle{amsalpha}

\end{document}